\documentclass[10pt,letterpaper]{article}
\usepackage[left=0.9in,right=0.9in,top=1in,bottom=1in]{geometry}
\usepackage{amsmath}
\usepackage{amsthm,alltt} 
\usepackage{amssymb}
\usepackage{graphicx}
\usepackage[onehalfspacing]{setspace}
\usepackage{multirow}
\usepackage{verbatim}
\usepackage{mathrsfs}
\usepackage{mathtools}
\usepackage[dvipsnames, table, xcdraw]{xcolor}
\usepackage{algpseudocode, algorithm}
\usepackage[normalem]{ulem}
\usepackage{indentfirst}
\usepackage{longtable}

\allowdisplaybreaks
\numberwithin{equation}{section}

\newtheorem{theorem}{Theorem}[section]
\newtheorem{proposition}{Proposition}[section]
\newtheorem{lemma}{Lemma}[section]
\newtheorem{definition}{Definition}[section]

\begin{document}

\title{Using Submodularity  within Column Generation to Solve the Flight-to-Gate Assignment Problem}
\author{Yijiang Li\thanks{School of Industrial and Systems Engineering, Georgia Institute of Technology (yijiangli@gatech.edu)}, John-Paul Clarke\thanks{Department of Aerospace Engineering and Engineering Mechanics, The University of Texas at Austin (johnpaul@utexas.edu)}, and Santanu S. Dey \thanks{School of Industrial and Systems Engineering, Georgia Institute of Technology (santanu.dey@isye.gatech.edu)}}
\date{}
\maketitle
\quad\\
\textbf{Abstract.}
In this paper, we provide a column generation-based approach for solving the airport flight-to-gate assignment problem, where the goal is to minimize the on-ground portion of arrival delays by optimally assigning each scheduled flight to a compatible gate. Specifically, we use a set covering formulation for the master problem and decompose the pricing problem such that each gate is the basis for an independent pricing problem to be solved for assignment patterns with negative reduced costs. We use a combination of an approximation algorithm based on the submodularity of the underlying set and dynamic programming algorithms to solve the independent pricing problems. To the best of our knowledge, this is the first use of submodularity property to efficiently solve pricing problems and improve the performance of column generation algorithm. We show that the dynamic programming algorithm is pseudo-polynomial when there are integer inputs. We also design and employ a rolling horizon method and block decomposition algorithm to solve large-sized instances. Finally, we perform extensive computational experiments to validate the performance of our approach.
\quad \\
\textbf{Keywords:} submodularity, column generation, flight-to-gate assignment, approximation algorithms, dynamic programming

\section{Introduction}
\subsection{Motivation and literature survey}
Airports throughout the world have seen a rapid increase in the numbers of flights and passengers over the past decade. Not withstanding the current decline due to COVID-19, the international Air Transport Association \cite{iata} expects $7.2$ billion passengers to travel in $2035$, a near doubling of the $3.8$ billion air travelers in $2016$. The vast majority of current airport facilities, in particular airport terminals, are not sized to handle such traffic. And, although expansions in capacity have been planned for the long term, in the near future the requisite capacity expansions are unlikely to materialize. Airports will therefore see increases in delays and associated increases in the cost of delays that are similar to or in excess of the $\$1.6$ billion or $6$ percent increase, from $\$26.6$ to $\$28.2$ billion, that was observed between 2017 and 2018 \cite{faa}.  These expected increases in the magnitude and cost of delays can only be mitigated through improved airport operations management.

One critical area of airport operations is the optimal and efficient assignment of arriving flights to gates. The airport flight-to-gate assignment problem has been extensively studied by many researchers in both the operations research and aviation communities. Many models and algorithms have been proposed. For a detailed review of past work in this area, we refer the readers to \cite{reviewsa} and \cite{DasGzaraStuzle}, but give a brief overview of the popular objectives and the common solution methods.

There are three popular objectives to consider in the models. The first one is the maximization of passenger satisfaction levels. In particular, researchers consider walking distances and waiting or transit time as proxies for the passenger satisfaction level. A more detailed approach involves a differentiation between the transfer passengers and destination passengers. In these models, the costs associated with the assignments of two flights to any two gates reflect the distance between the two gates and the connection time between the two flights. Interested readers are referred to \cite{mokhtarimousavi}, \cite{ChengHoKwan}, \cite{YanTang}, and \cite{kimFeronClarke}. The second class of objective focuses on airport and airline operations. Some papers, such as \cite{Kaliszewksi} and \cite{Cheng}, consider the minimization of the number of ungated flights or equivalently the number of flights assigned to the remote gates or introduce costs only incurred when the flights are assigned to remote gates to differentiate the priorities of the flights. Moreover, \cite{Jaehn} and \cite{NeumanAtkin} point out that airlines likely have preferences over which set of gates to park their flights leading to the consideration of maximizing total gate preference scores. The third popular objective is to improve the robustness of the solution to schedule variations. Due to the uncertain nature of the air transportation system, arrival times and departure times are likely to be stochastic subject to various factors such as weather and maintenance. Early arrival and late departure may cause the arriving flight to wait for the departing flight at the assigned gate to push back. This phenomenon is referred to as a gate conflict. Other studies consider robustness with respect to minimizing the expected gate conflict time. In the event of major disruptions such as severe weather or maintenance delays, airport has to be able to quickly recover from the deviations from the original arrival schedules and adjust the flight-to-gate assignments to reduce any potential delays. Readers are referred to the work of \cite{DorndorfJaehnPesch} and \cite{KumarBielaire}. Since the above mentioned three objectives are all very important to take into account and neither of the objectives is superior than the other two, some combinations of them are also explored. The weighted sum method is more common in such line of considerations, but a pareto-based approach has been utilized as well. The papers \cite{KimFeronClarkeMarzuoliDelahaye} and \cite{DorndorfJaehnPeschReferenceSchedule} are examples of studies that use weighted sum approach while \cite{Das} approaches this problem with a pareto local search method.

Due to the general $\mathcal{NP}$-hard nature of the problem, it can be extremely difficult to solve the problems directly and exactly. Many solution algorithms and heuristics have been proposed. In much of the work to date, an integer programming or mixed integer programming formulation is first presented and solver is called to obtain a solution and validate the model. These formulations are large and complex because of the many complicated constraints required to reflect real world phenomena. Consequently, it is very time-consuming if not impossible to obtain solutions within reasonable amount of time and it is challenging to improve these solution methodologies. Readers are referred to \cite{KumarBielaire} and \cite{TangWang} for the details of these models. In addition to integer programming or mixed integer programming approach, \cite{SekerNoyan} uses a stochastic optimization model to capture the uncertain nature of the problem and study the occurrence of gate conflicts. The paper \cite{YanHuo} uses a column generation method together with a branch-and-bound algorithm to obtain integral solutions. They give a description of the column generation method and a general strategy to apply the branch-and-bound algorithm. However, they have not given any explicit mathematical formulations for the column generation scheme and they have not discussed any strategies or methods to tackle the pricing problems. Due to the limitations/challenges presented in using exact formulations for this problem, heuristics are extensively used. Popular heuristics such as genetic algorithms and tabu search algorithms have been explored. Tabu search algorithm performs a local search around the current solution and prohibits searching at previous search points by keeping a list of those points for a few iterations while genetic algorithm is inspired from the evolutionary idea where a chromosome of higher fitness is more likely to survive and it involves operations of selections, crossover, and mutations. For an implementation of these two metaheuristics, readers are referred to \cite{kimFeronClarke}. Moreover, \cite{YuZhangLau} implemented a MIP-based heuristic which is adapted from many popular heuristic and they show that the adaptation gives a good performance. Lastly, there are some other heuristics based on neighborhood search, breakout local search, Fuzzy Bee Colony optimization (FBCO), and particle swarm optimization (PSO). Readers are referred to \cite{YuZhangLauNeighborhood}, \cite{BenlicBurkeWoodward}, \cite{BenlicBurkeWoodward}, and \cite{DengZhaoYangXiongSunLi} for these heuristics.
\subsection{Contributions of this paper}
As we discussed previously, there are gaps in both the formulations and solution methodologies of the flight-to-gate assignment problem that need to be addressed. In this paper, we focus on the airport operations and aim to minimize the arrival delays over all the flights. We propose an exact solution method that is capable of obtaining the flight-to-gate assignments at busy hub airports where a large number of flights are seen on a daily basis. Our main contributions are summarized below. 
\begin{enumerate}
	\item \textbf{Column generation formulation.} We present an explicit column generation formulation with a set covering master problem and we decompose the pricing problem into independent problems to which we apply both heuristic and exact algorithms to obtain favorable assignment patterns. Such decomposition allows us to simplify the pricing problem formulations and reduce the size of the individual pricing problem significantly. Experiments show that this formulation is far more efficient than a conventional compact mixed integer programming formulation.  
	\item \textbf{Pricing problem algorithms.} We explore a few approximation methods and exact algorithms to efficiently solve individual pricing problems after the decomposition. In particular, 
	\begin{itemize}
	\item We show and exploit the submodularity of the objective functions of the pricing problems to design an approximation algorithm. To the best of our knowledge, this is the first use of submodularity to efficiently solve the pricing problems in a column generation setting. 
	\item We derive an exact dynamic programming algorithm based on a recursive formula and show that the direct implementation of the algorithm works reasonably well in practice. We show the algorithm to be pseudo-polynomial in the special case of integer inputs and present a new heuristic based on this observation.
	\item We propose a block decomposition heuristic and prove it is a $2$-approximation algorithm for large-sized instances.
	\end{itemize}
\end{enumerate}

The structure of the rest of the paper is as follows. In Section \ref{problemSetup}, we provide a description of the flight-to-gate assignment problem. In Section \ref{formulation}, we provide the formulation for the column generation algorithm, and the exact algorithms and heuristics used for solving the pricing problems are discussed in Section \ref{solvingpricing}. Specifically, a submodular maximization approximation algorithm is described in Subsection \ref{approximation} and a dynamic programming algorithm with complexity analyses is presented in Subsection \ref{DP}. In Subsection \ref{largesizedinstance}, we present a new block decomposition approximation for large-sized instances. We discuss the branching scheme and our method for retrieving integer feasible solutions in Section \ref{branching}. Results of the computational experiments are presented in Section \ref{numericalexp}. Finally, we make some concluding remarks and discuss possible extensions of this work in Section \ref{conclusions} . 

\section{Problem setup} \label{problemSetup}
The flight-to-gate assignment problem may be described as follows. Let the set of flights be $\mathcal{F}$  and the set of gates be $\mathcal{G}$. Each flight $i \in \mathcal{F}$ has an arrival time, $a_i$, an airline, and an aircraft type which determines its minimum turn time, $\tau_i$, the minimum duration for which the flight must remain after parking at its assigned gate before being ready for departure. Each gate $k \in \mathcal{G}$ has a buffer time, $b_k$, that has to be observed between consecutive flights, and a set of allowable aircraft types. The solution of the problem gives an assignment in which each flight is assigned to one gate. For each flight $i$, we define its arrival delay as the difference between its arrival time, $a_i$ and the park time at its assigned gate, $t_i^g$. The decisions are the gate assignment, $x_{ik}$, the park time, $t_i^g$, and the push back time, $t_i^p$, for $i \in \mathcal{F}$ and $k \in \mathcal{G}$. In particular, the decision variables $x_{ik}$ are equal to $1$ if flight $i$ is assigned to gate $k$ and otherwise equal to $0$. In addition, we also consider the notion of compatibility, $\alpha_{ik}$, where a flight can only be assigned to a compatible gate. In particular, $\alpha_{ik}$ equal to $1$ if flight $i$ is compatible to gate $k$ and equal to $0$ otherwise. The compatibilities are determined by various factors including gate types (heavy or regular), flight types (heavy or regular), and airline preferences. Two additional conditions have to be met for an assignment to be valid:
\begin{enumerate}
	\item A flight occupies its assigned gate for a duration of at least its minimum turn time, $\tau_i$, for $i \in \mathcal{F}$.
	\item A pair of flights that are assigned to the same gate cannot occupy the gate at the same time and the buffer time, $b_k$, must be observed before the gate becomes ready for the next flight.
\end{enumerate}

We further assume independence of the gates which exclude interferences between the gates. This assumptions can be justified from a practical perspective. Note first that we are only considering arrival delays in this paper. The current practices of many airports is that pushbacks of the departing flights are coordinated by the ramp towers. Whenever a pushback of a flight is in the way of an arriving flight into its assigned gate, the pushback is delayed. Note that such delays are rare in many of the large airport where there are multiple taxi ways in the ramp area to minimize blockage of arriving flights by the departing flights. Consequently, what affects the park time of an arriving flight most is its assigned gate and the departing flight at the assigned gate. Based on the above descriptions, we give a compact mixed integer programming formulation below,
\begin{align}
\min \; & \sum_{i \in \mathcal{F}} (t_i^{g}-a_i) \label{MIPO1}\\
\text{subject to } & \sum_{k \in \mathcal{G}} x_{ik} = 1, \; \forall i \in \mathcal{F} \label{MIPC1}\\
& t_i^{g} + \tau_i \leq t_i^{p}, \; \forall i \in \mathcal{F} \label{MIPC2}\\
& t_i^{g} \geq a_i, \; \forall i \in \mathcal{F} \label{MIPC3}\\
& t_i^{p} + b_k - t_j^{g} \leq M(2-x_{ik}-x_{jk}), \; i < j,\; \forall i, j \in \mathcal{F},\; \forall k \in \mathcal{G} \label{MIPC4}\\
& x_{ik} \leq \alpha_{ik}, \; \forall i \in \mathcal{F}, j \in \mathcal{G} \label{MIPC6}\\
& x_{ik} \in \{0,1\},\; \forall i \in \mathcal{F}, k \in \mathcal{G} \label{MIPbinary}\\
& t_i^{g}, t_i^{p}\ge 0, \; \forall i \in \mathcal{F}, \label{MIPcontin}
\end{align} 
where $M$ is a sufficiently large constant. The expression in (\ref{MIPO1}) gives the objective function which minimizes the total arrival delays. Constraint (\ref{MIPC1}) ensures that each flight is assigned to exactly one gate. Constraint (\ref{MIPC2}) ensures the flight parks at a gate for at least a duration of the minimum turn time (condition 1 above). Constraint (\ref{MIPC3}) ensures that the park time, $t_i^g$, is no earlier than the arrival time, $a_i$. We want to emphasize that we do not impose an upper bound on the park times, $t_i^g$, since it is possible that an arriving flight has to wait at its assigned gate for a long period of time for the departing flight that is currently occupying the gate to pushback due to various reasons, such as weather, technical difficulties, or security reasons, although the chance of such long wait is slim. In addition, imposing such upper bounds on the park times can lead to infeasibility sometimes. Condition 2 above leads to constraint (\ref{MIPC4}) which observes the buffer time, $b_k$. It considers all pairs of flights $i < j$. When they are assigned to the same gate, the difference between $t_i^{p}$ and $t_j^{g}$ must be at least the buffer time of the gate, $b_k$. Constraint (\ref{MIPC6}) is the compatibility constraint to ensure the flights are only assigned to the compatible gates. Constraints (\ref{MIPbinary}) and (\ref{MIPcontin}) are binary and non-negativity requirements on the decision variables respectively. We will show in computational experiments that this compact formulation is not ideal to obtain the desired flight-to-gate assignment efficiently. We instead propose a column generation approach to solve the problem. 
Note also that this setup is in fact similar to what airlines do for their flight-to-gate assignments in the real world. Interested readers are referred to \cite{deltatabu}.

\section{Column generation formulation} \label{formulation}
\subsection{Master problem consideration}
Constraint (\ref{MIPC1}) in the compact mixed integer programming formulation suggests a set partitioning master problem should be used in the column generation formulation. However, as many papers, such as \cite{ColGen} and \cite{VRPTW}, have pointed out, a set covering master problem formulation is numerically more stable when solving its linear programming relaxation compared to a set partitioning master problem. Therefore, we present a set covering master problem formulation where the constraint in which each flight is assigned to exactly one gate is replaced by a constraint in which each flight is assigned to at least one gate. Since this gives a relaxation and we are minimizing the total arrival delay that is non-increasing after removing a flight from a gate, the set covering master problem either produces a set partitioning assignment or a set covering assignment in which we can fix each flight with multiple assignments to one of its assigned gates and recover a set partitioning assignment that has a total arrival delay at most as large as that of the original set covering assignment.
\subsection{The set covering master problem}
We define the following parameters:
\begin{align*}
	P_k & := \text{ the set of all feasible assignment patterns for gate }k\\
	\delta_{ip}^k & := \begin{cases}1 \; \text{ if flight } i \text{ is assigned on pattern } p \in P_k \text{ of gate }k \\ 0 \; \text{ otherwise}\end{cases}\\
	c_p^k & := \text{ the arrival delay of pattern } p \in P_k.
\end{align*}
Note that $\delta_{ip}^k$ and $c_p^k$ are constants that can be computed for a given assignment pattern. The decision variables $z_p^k \;(k \in \mathcal{G}, p \in P_k)$ are equal to $1$ if pattern $p$ of gate $k$ is used and $0$ otherwise. Then the set covering master problem is given by
\begin{align}
\min \; & \sum_{k \in \mathcal{G}} \sum_{p \in P_k} c_p^k z_p^k \label{colGenObj}\\
\text{subject to } & \sum_{k \in \mathcal{G}} \sum_{p \in P_k} \delta_{ip}^k z_p^k \ge 1, \; \forall i \in \mathcal{F} \quad (\pi_i)\label{coverConstraint}\\
& \sum_{p \in P_k} z_p^k = 1, \; \forall k \in \mathcal{G} \quad (\mu_k)\label{availConstraint}\\
& z_p^k \in \{0,1\}, \; \forall k \in \mathcal{G}, p \in P_k. \label{bound}
\end{align}

The objective (\ref{colGenObj}) is an expression for total arrival delays by summing the arrival delays of all patterns over all gates. Constraint (\ref{coverConstraint}) is the cover constraint that ensures each flight is assigned to at least one gate. Constraint (\ref{availConstraint}) is the availability constraint that ensures only one pattern of each gate is selected. Last constraint (\ref{bound}) is the binary requirements on the decision variables $z_p^k \;(k \in \mathcal{G}, p \in P_k)$ and relaxing these constraints to $0 \le z_p^k \le 1$ gives the linear programming relaxation of the set covering master problem. The linear programming relaxation is then solved in each iteration. 
\subsection{The pricing problems}
In the pricing problem, we construct feasible assignment patterns for each gate. If we associate dual variables $\pi_i$ with (\ref{coverConstraint}) and $\mu_k$ with (\ref{availConstraint}), the reduced cost $\bar{c}_p^k \;(k \in \mathcal{G}, p \in P_k)$ of a pattern $p$ of gate $k$ is given by 
\begin{equation}
\bar{c}_p^k = c_p^k - \sum_{i \in \mathcal{F}} \delta_{ip}^k\pi_i - \mu_k.
\end{equation}

Given a feasible solution $z$ to the linear programming relaxation of (\ref{coverConstraint}) - (\ref{bound}), we know that $z$ is optimal if and only if the optimal assignment patterns generated from the pricing problem have non-negative reduced costs. Then the pricing problem is given by
\begin{equation}
\min\{c_p^k - \sum_{i \in \mathcal{F}} \delta_{ip}^k\pi_i - \mu_k: k \in \mathcal{G}, p \in P_k\}.
\end{equation}

Since we assume each gate is independent, we can further decompose the pricing problem into $|\mathcal{G}|$ independent pricing problems, one for each gate. The $k^{th}$ pricing problem is given by 
\begin{align} \label{pricingproblem}
\min & \{c_p^k - \sum_{i \in \mathcal{F}} \delta_{ip}^k\pi_i - \mu_k: p \in P_k\}.
\end{align}

For the decision variables, similar to the compact formulation (\ref{MIPO1}) - (\ref{MIPcontin}), we use binary decision variables $x_{ik}$ which are equal to $1$ if flight $i$ is assigned at gate $k$ and equal to $0$ otherwise, and continuous variables $t_i^g$ and $t_i^{p}$ for the park times and push back times of flight $i$ respectively. Recall the definitions of the constants $\delta_{ip}^k$ and $c_p^k$ which are given by
\begin{align*}
	\delta_{ip}^k & := \begin{cases}1 \; \text{ if flight } i \text{ is assigned on pattern } p \in P_k \text{ of gate }k \\ 0 \; \text{ otherwise}\end{cases}\\
	c_p^k & := \text{ the arrival delay of pattern } p \in P_k.
\end{align*}

When the pricing problems are solved to generate a new assignment pattern $p$, the values of the constants, $\delta_{ip}^k$, are determined using the values of the decision variables $x_{ik}$ and $c_p^k$ is the sum of all individual arrival delay, $t_i^g-a_i$. Let $M$ denote a large constant, then the $k^{th}$ pricing problem is given by
\begin{align}
\min \; & \sum_{i \in \mathcal{F}} (t_i^{g}-a_i) - \sum_{i \in \mathcal{F}} x_{ik}\pi_i - \mu_k \label{pricingObj}\\
\text{ subject to } \; & t_i^g + \tau_i \leq t_i^{p}, \; \forall i \in \mathcal{F} \label{C1}\\
& t_i^{g} \geq a_i, \; \forall i \in \mathcal{F} \label{C2}\\
& t_i^{p} + b_k - t_j^{g} \le M(2-x_{ik}-x_{jk}), \; \forall i < j,\;\;\; i,j \in \mathcal{F}\label{C3}\\
& x_{ik} \le \alpha_{ik}, \; \forall i \in \mathcal{F}, k \in \mathcal{G} \label{compatibility}\\
& x_{ik} \in \{0,1\}, \; \forall i \in \mathcal{F},\; k \in \mathcal{G} \label{assBound}\\
& t_i^{g}, t_i^{p}\ge 0, \; \forall i \in \mathcal{F}. \label{timeBound}
\end{align}

The objective function (\ref{pricingObj}) is obtained by expressing $c_p^k$ explicitly in decision variables $t_i^g$ and parameters $a_i$, and substituting the parameters $\delta_{ip}^k$ by the decision variables $x_{ik}$. Constraints (\ref{C1}) - (\ref{timeBound}) follow directly from the constraints (\ref{MIPC2}) - (\ref{MIPcontin}) in the compact mixed integer programming formulation. However, we do not need constraint (\ref{MIPC1}). For flights that are not assigned to this gate, we assume their arrival delays to be zero.

The solutions $x_{ik}$ from a pricing problem form a potential assignment pattern. If the corresponding reduced cost given by the objective value is negative, then this assignment pattern is favorable and added to the master problem. 

We want to point out that $\mu_k$ is a constant in the $k^{th}$ pricing problem and thus dropping this term we can equivalently solve (\ref{pricingObj}) - (\ref{timeBound}) by solving the following maximization problem:
\begin{align}
\max \; & - \sum_{i \in \mathcal{F}} (t_i^g - a_i) + \sum_{i \in \mathcal{F}} x_{ik}\pi_i \label{modifiedpricingObj}\\
\text{ subject to } \;& (\ref{C1}) - (\ref{timeBound}). \notag
\end{align}

This equivalent problem is intuitively easier to interpret as we may consider $\pi_i$ as the benefits of accepting flight $i$ at the cost of an arrival delay $t_i^g - a_i$. Consequently, we are attempting to maximize the total net benefits over all flights. We refer to this maximization problem as the pricing problem for any gate in all following sections and use the terms total net benefits and reduced cost interchangeably. 

\section{Solving the pricing problem} \label{solvingpricing}
\subsection{Pre-processing}
The decomposition of the pricing problem into smaller pricing problems allows additional pre-processing of the set of flights $\mathcal{F}$ for each gate. Consider a particular gate $k$, we can construct a subset of flights $\mathcal{F}' \subseteq \mathcal{F}$ to include only compatible flights and flights with positive dual variables $\pi_i$. To see this fixing is correct, the incompatible flights are not allowed to be assigned to this gate while accepting the flights with $\pi_i \le 0$ does not contribute any benefit at the cost of arrival delays of other flights. For the rest of the discussions in this section, we use the subset of flights $\mathcal{F}'$ and assume $|\mathcal{F}'|=n$. Note that for different gates, the set $\mathcal{F}'$ could be different.
\subsection{Submodular approximation algorithm} \label{approximation}
Submodularity is a well studied property of set functions and there are many existing algorithms that can approximate the task of maximizing a submodular function very efficiently and effectively. We notice that, in the context of the flight-to-gate assignment problem, the submodularity of the objective function (\ref{modifiedpricingObj}) is very intuitive. Consider any gate and given a set of accepted flights at this gate, the impact of accepting an additional flight on the total arrival delay of this set of flights is at least as large as that of accepting the additional flight with a subset of those flights.

Formally consider the pricing problem for gate $k$ and recall that our objective after pre-processing the set of flights is
\begin{align*}
\max \; & - \sum_{i \in \mathcal{F}'} (t_i^g - a_i) + \sum_{i \in \mathcal{F}'} x_{ip}\pi_i. 
\end{align*}

For simplicity, we denote $t_i^{g}-a_i$ by $\triangle t_i^{g}$ and clearly $\triangle t_i^g \ge 0$. We define a function $f: 2^{\mathcal{F}'} \mapsto \mathbb{R}$ as follows,
\begin{align} \label{subfDef}
f(A) = - \sum_{i \in A} \triangle t_i^{g, A} + \sum_{i \in A}\pi_i.
\end{align} 

We further denote the variables $\triangle t_i^{g}$ under a set of assigned flights $A$ by an additional superscript $A$ and assume that the values of $\triangle t_i^{g,A}$ for $i \notin A$ to be zero. By the definition of $f$, we have that $f(\emptyset) = 0$. For a given set of assigned flights $A$, the constraints (\ref{C1}) - (\ref{C3}) can be incorporated when evaluating $f(A)$. Specifically, let the flights in the set $A$ be $1, 2, \ldots\,|A|$ and the total benefits $\sum_{i \in A}\pi_i$ can be computed. What remains is to compute the arrival delays. Note that the first flight parks at the gate at its arrival time, $a_{1}$, and pushes back after a duration of its minimum turn time at $a_{1} + \tau_{1}$. The gate becomes available for the second flight after a duration of the buffer time at $a_{1} + \tau_{1} + b_k$.  Consequently, an arrival delay of $a_{1} + \tau_{1} + b_k - a_{2}$ on the second flight is incurred if the second flight arrives earlier than $a_{1} + \tau_{1} + b_k$ and it parks at $a_{1} + \tau_{1} + b_k$. Otherwise, there is no arrival delay on the second flight and it parks at its arrival time, $a_{2}$. The same procedure can be applied to the rest of flights in the set $A$ to compute $f(A)$.

Note that since the subset of flights $\mathcal{F}'$ are all compatible to this gate and the above procedure incorporates constraints (\ref{C1}) - (\ref{C3}), we can pose the pricing problems as the following equivalent unconstrained maximization problem 
\begin{equation} \label{submaxpro}
\max\{f(A): A \subseteq 2^{\mathcal{F}'}\}. \tag{subMax}
\end{equation}

In order to show the submodularity of the function $f$, we need the following observation about the variables $\triangle t_i^{g}$.
\begin{lemma}\label{subsetDelayLemma}
If $A \subseteq B \subseteq \mathcal{F}'$, then
$$\sum_{i \in B} \triangle t_i^{g, B} \ge \sum_{i \in A} \triangle t_i^{g, A}.$$
\end{lemma}
\begin{proof}[proof of Lemma \ref{subsetDelayLemma}.]
We can rewrite the left hand side expressions into
\begin{align*}
\sum_{i \in B} \triangle t_i^{g, B} & = \sum_{i \in A} \triangle t_i^{g, B} + \sum_{i \in B \backslash A} \triangle t_i^{g, B}\\
& \ge \sum_{i \in A} \triangle t_i^{g, A} + \sum_{i \in B \backslash A} \triangle t_i^{g, B}.
\end{align*}

The last inequality is valid because $A \subseteq B$, for any $i \in A$, thus the arrival delay of $i$ in the set of assigned flights $B$ of larger cardinality is at least as large as its arrival delay in the set of assigned flights $A \subseteq B$. Additionally, since $\triangle t_i^{g,B} \ge 0$, we have the desired inequality. 
\end{proof}

We next show that $f$ is a submodular function. 
\begin{lemma} \label{submodFunc}
$f$ is a submodular set function defined on $\mathcal{F}'$. 
\end{lemma}
\begin{proof}[proof of Lemma \ref{submodFunc}.]
First note that we use the expression ``the park time of a flight is pushed to the right" to mean the park time of this flight is delayed and consequently this flight experiences a larger arrival delay. Formally, given a set of flights $A$, the park time of flight $i$ is pushed to the right in another set of flight $B$ if $\Delta t_i^{g,A} \le \Delta t_i^{g,B}$.

Now suppose that $A \subseteq B \subseteq \mathcal{F}'$ and $u \notin \mathcal{F}' \backslash B$, we compare $f(A \cup \{u\}) - f(A)$ against $f(B \cup \{u\}) - f(B)$, 
\begin{align}
f(A \cup \{u\}) - f(A) & = - \sum_{i \in A \cup \{u\}} \triangle t_i^{g, A \cup \{u\}} + \sum_{i \in A \cup \{u\}}\pi_i  - \left( - \sum_{i \in A} \triangle t_i^{g, A} + \sum_{i \in A}\pi_i \right) \notag\\
& =  - \left(\sum_{i \in A} \triangle t_i^{g, A  \cup \{u\}} - \sum_{i \in A} \triangle t_i^{g, A}\right) + (\pi_u - \triangle t_u^{g, A \cup \{u\}}). \label{subA}
\end{align}

Similarly, we have that 
\begin{align}
f(B \cup \{u\}) - f(B) = - \left(\sum_{i \in B} \triangle t_i^{g, B \cup \{u\}} - \sum_{i \in B} \triangle t_i^{g, B}\right) +  (\pi_u- \triangle t_u^{g, B \cup \{u\}}). \label{subB}
\end{align} 

Since $A \subseteq B$, the arrival delay of $u$ in the set of assigned flights $B$ is at least as large as that of $u$ in the set of assigned flights $A$ and then $\triangle t_u^{g, A \cup \{u\}} \le \triangle t_u^{g, B \cup \{u\}}$. Thus $\pi_u - \triangle t_u^{g, A \cup \{u\}} \ge \pi_u- \triangle t_u^{g, B \cup \{u\}}$. For the first term in the expressions (\ref{subA}) and (\ref{subB}), we consider the following cases when accepting an additional flight $u$ to $A$ and $B$:
\begin{enumerate}
  \item The park times remain unchanged for all $i \in B$ and thus $i \in A$. Both terms are zero. We have $f(A \cup \{u\}) - f(A) \ge f(B \cup \{u\}) - f(B)$.
  \item The park times are pushed to the right for a set of flights $I \subseteq B$. If $I \cap A \neq \emptyset$, following similar arguments for $u$, we see that the increase in the arrival delay of a flight $i \in I \cap A$ in the set of assigned flights $B$ is at least as large as that of $i$ in the set of assigned flights $A$. In addition, there are potentially increases in arrival delays of flights in the set $I \cap A^c$. Therefore, the first term in (\ref{subA}) is less negative than that in (\ref{subB}). If $I \cap A = \emptyset$, the first term in (\ref{subA}) is zero while the first term in (\ref{subB}) is non-positive. We have again $f(A \cup \{u\}) - f(A) \ge f(B \cup \{u\}) - f(B)$.
\end{enumerate}

The two cases above verify that when we accept an additional flight $u$, we have $f(A \cup \{u\}) - f(A) \ge f(B \cup \{u\}) - f(B)$ and this shows that $f$ is a submodular function.  
\end{proof}

We want to point out that $f$ is not monotonic in general. Suppose that $A \subseteq B \subseteq \mathcal{F}'$, then 
\begin{align*}
f(B) - f(A) & = - \sum_{i \in B} \triangle t_i^{g, B} + \sum_{i \in B}\pi_i - \left(-\sum_{i \in A} \triangle t_i^{g, A} + \sum_{i \in A}\pi_i\right)\\
& =  - \left(\sum_{i \in B} \triangle t_i^{g, B} - \sum_{i \in A} \triangle t_i^{g, A}\right) + \sum_{i \in B \backslash A} \pi_i.
\end{align*}

The first term $- \left(\sum_{i \in B} \triangle t_i^{g, B} - \sum_{i \in A} \triangle t_i^{g, A}\right)$ is always non-positive by Lemma \ref{subsetDelayLemma}. However, since $\pi_i \ge 0$ for all $i \in \mathcal{F}$ and  $\sum_{i \in B \backslash A} \pi_i \ge 0$, $f(B) - f(A)$ is not necessarily always non-negative or non-positive and thus $f$ is not monotone in general.

Now we can utilize an existing submodular maximization algorithm by \cite{subMax} which is shown below as algorithm \ref{SubmodularApproximation}. In this algorithm, we select each flight in the set of flights $\mathcal{F}'$ with a probability. Let $X_i$ and $Y_i$ be two random sets to be updated in each iteration and we initialize $X_0 = \emptyset$ and $Y_0 = \mathcal{F}'$. We also denote the elements in $\mathcal{F}'$ by $1,2,\ldots,n$.
\begin{algorithm}[H]
\caption{Submodular Maximization($f,\mathcal{F}'$)}\label{SubmodularApproximation}
\begin{algorithmic}[1]
\State $X_0\gets \emptyset, Y_0\gets \mathcal{F}'$
\For{$i = 1,2,\ldots,n$}
\State $a_i\gets f(X_{i-1} \cup \{i\}) - f(X_{i-1})$
\State $b_i\gets f(Y_{i-1} \backslash \{i\}) - f(Y_{i-1})$
\State $a_i^{\prime}\gets \max\{a_i,0\}, b_i^{\prime}\gets \max\{b_i,0\}$
	\State \textbf{with probability} {$a_i^{\prime} / (a_i^{\prime} + b_i^{\prime})^*$} \textbf{do:} \Comment{$^*$ If $a_i^{\prime} = b_i^{\prime} = 0$, we assume $a_i^{\prime} / (a_i^{\prime} + b_i^{\prime}) = 1$}
		\State $X_i\gets X_{i-1} \cup \{i\}, Y_i\gets Y_{i-1}$, 
	\State \textbf{else}\; (with the compliment probability $b_i^{\prime} / (b_i^{\prime} + a_i^{\prime}$))
		\State $X_i\gets X_{i-1}, Y_i\gets Y_{i-1} \backslash \{i\}$
	\State \textbf{end}
\EndFor
\State \Return{$X_n$ (or equivalently $Y_n$)}
\end{algorithmic}
\end{algorithm}

The following theorem establishes a theoretical approximation guarantee of Algorithm \ref{SubmodularApproximation}.
\begin{theorem} \label{subguarantee}
	Let $f: 2^{\mathcal{F}'} \mapsto \mathbb{R}$ defined by (\ref{subfDef}), and $OPT$ be the true optimal solution of the problem (\ref{submaxpro}) and $X_n$ (or equivalently $Y_n$) is the set returned by the algorithm. If $f(\mathcal{F}') \ge 0$, then $\mathbb{E}(f(X_n)) \ge f(OPT)/2$.
\end{theorem}
\begin{proof}[proof of Theorem \ref{subguarantee}.]
The proof is a very minor modification of what is given in \cite{subMax} and attached in Appendix A for a reference and completeness. The modification is since we require $f(\mathcal{F}') \ge 0$ instead of $f(A) \ge 0$ for all $A \subseteq \mathcal{F}'$. 
\end{proof}

Although it is possible that the condition on $f$ to obtain the theoretical guarantee is not satisfied, the assignment patterns that are generated can still have favorable total net benefits. We discuss later in the numerical experiment section about the implementation of this algorithm and its performance. 

\subsection{Dynamic programming algorithm} \label{DP}
Approximation algorithm can be efficient in generating favorable assignment patterns, however, exact optimal solutions to the pricing problems have to be obtained to prove optimality at each node. This would typically be done using a general integer programming solver. However, this can be slow. To reduce overall solution time, we have developed and now present a much more effective dynamic programming algorithm to achieve this task. We divide up the analyses of the dynamic programming algorithm into three cases based on the input type, namely, general input, rational input, and integer input. We note that the analysis in the case of integer input is a simplification of the analysis used in the case of rational input. In addition, although in many cases, the input is not integer, the algorithm developed in that section can serve as another approximation algorithm. Therefore, we defer the details of the analysis in case of rational input to Appendix B.
\subsubsection{The general case.}
For any gate $k$, we define an auxiliary function $g_i(t)$ to be the maximum total net benefits from optimally accepting flights in the set $\{i,i+1, \cdots, n\}$ where the smallest indexed accepted flight from $\{i,i+1,\ldots,n\}$ can be accepted any time after time $t$. To give a formal definition, 
\begin{equation} \label{f(i,t)def}
g_i(t) = \max_{x_{ik}, x_{i+1,k}, \ldots, x_{nk}} \left\{\sum_{j \ge i} x_{jk} \pi_j - \sum_{j \ge i} (t_j^g - a_j)\; |\;t_j^g \ge t, \;\forall j \in \{i,i+1,\ldots,n\} \right\}.
\end{equation}

As an illustration, we present an example of a set of $8$ arrivals and show in Figure \ref{func_g} the functions $g_i(t)$. In particular, the left plot shows the function $g_1(t)$ while the right plot shows the function $g_8(t)$. \\
\begin{figure}[!tbhp]
\centering
\includegraphics[width=0.9\textwidth]{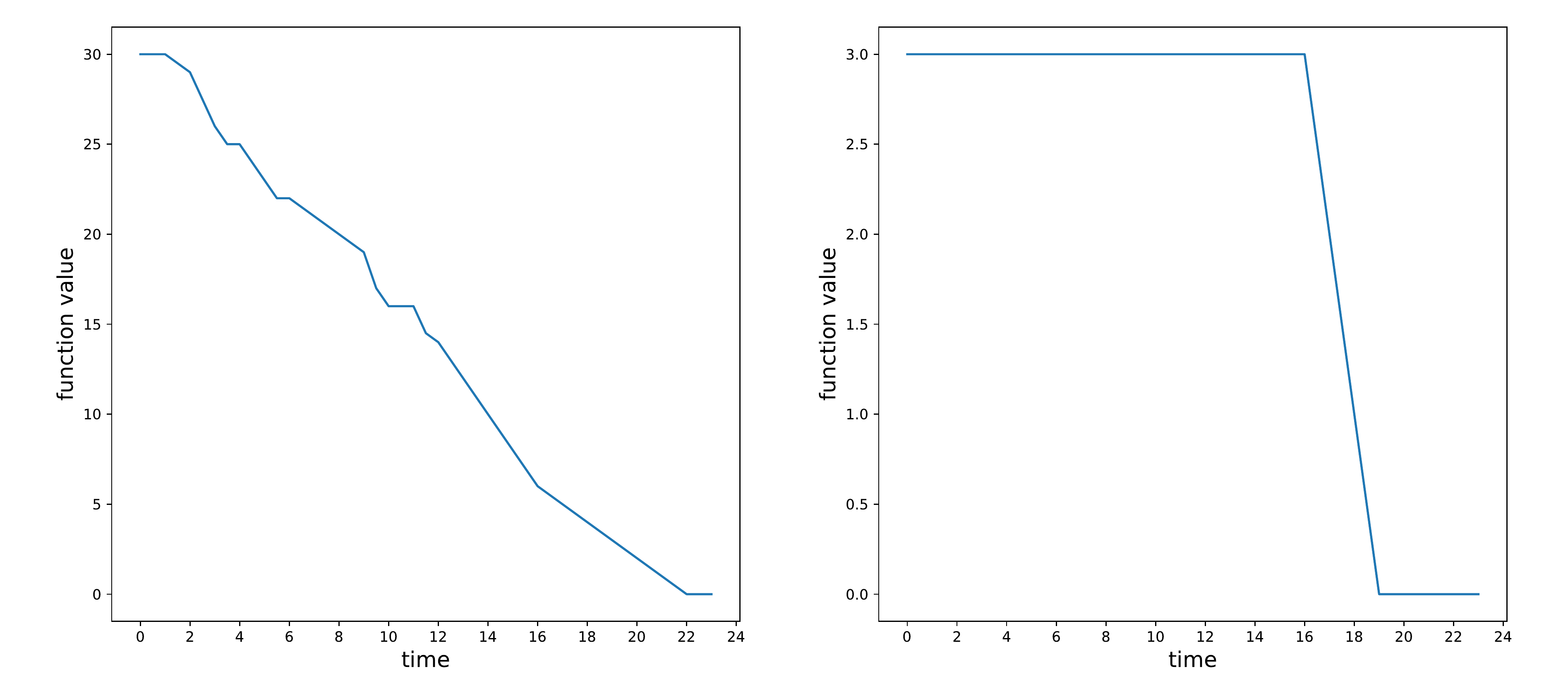}
\caption{Plot of functions $g_1(t)$ and $g_8(t)$.\label{func_g}}
\end{figure}

We observe that the optimal value of the $k^{th}$ pricing problem is $g_1(0)$. We introduce the notion of processing time of a flight $i$ at gate $k$ which consists of the minimum turn time and the buffer time and denote the processing time of flight $i$ by $p_i := \tau_i + b_k$. In the rest of this section, we drop the index $k$ as we are solving the pricing problem for a fixed gate $k$. Based on the definition of $g_i(t)$, we can derive the following recursive formula,
\begin{lemma}\label{recurlemma}
\begin{align}\label{recureq}
g_i(t) = \begin{cases}
0 & \; \text{ if } i \ge n+1\\
g_i(a_i) & \; \text{ if } t \le a_i\\
g_{i+1}(t) & \; \text{ if } t > a_i + \pi_i\\
\max\{(a_i + \pi_i - t) + g_{i+1}(t + p_i), g_{i+1}(t)\} &\; \text{ if }a_i < t \le a_i + \pi_i.
\end{cases} 
\end{align}
\end{lemma}
\begin{proof}[proof of Lemma \ref{recurlemma}.]
Similar to the proof of Lemma \ref{submodFunc}, we use the expression ``the park time of a flight is pushed to the right" to mean the park time of this flight is delayed and consequently this flight experiences a larger arrival delay.

We consider each case separately.
\begin{enumerate}
\item This is the terminating case. If $i \ge n+1$, we do not have any flights and thus receive a zero net benefit. 
\item Since a flight cannot park earlier than its arrival time, the total net benefits for any $t \le a_i$ are equal to $g_i(a_i)$. 
\item If the current time $t > a_i + \pi_i$, accepting the flight $i$ does not contribute to the total net benefits as $\pi_i - (t - a_i) < 0$. In addition, the park times of flights in the set $\{i+1,i+2,\cdots,n\}$ are potentially pushed to the right in the presence of flight $i$ and their arrival delays are at least as large as the arrival delays in the absence of flight $i$. Consequently, we do not accept flight $i$ at this gate. From the above analyses, we see that $a_i + \pi_i$ is the latest time beyond which we do not accept flight $i$. We denote this time as the end point and $a_i$ as the start point of flight $i$'s acceptance window.
\item In this last case, we compare the total net benefits between accepting and not accepting flight $i$. In the former case, if we accept flight $i$ at the current time $t$, the net benefit gained from flight $i$ is $a_i + \pi_i - t$ and earliest possible park time for flight $i+1$ is $t+p_i$. In the latter case, if we do not accept flight $i$, the earliest possible park time for flight $i+1$ is $t$. 
\end{enumerate} 
\end{proof}
A direct implementation of the formula (\ref{recureq}) recursively considers whether to accept each flight for all flights starting with the first flight. 
Algorithm \ref{DPpricing} below shows an implementation to obtain the value of $g_1(0)$. Note that since we define the function $g_i(t)$ for all continuous values of $t$, it was not immediately clear whether Algorithm $2$ would run in finite time. Thus we verify that Algorithm \ref{DPpricing} runs in finite time next.
\begin{proposition}\label{finitealgorithm}
For any value of $i$ and $t$, Algorithm \ref{DPpricing} runs in finite time. In particular, to evaluate $g_i(t)$ for any $t$, the procedure $\textup{EVALORACLE}(i,t)$ is recursively called $2^{(n-i+1)}$ times in the worst case.
\end{proposition}
\begin{proof}[proof of Proposition \ref{finitealgorithm}.]
The proof is by backward induction. If $i = n$, evaluating the function $g_n(t)$ for any $t$ is equivalent to the procedure $\textup{EVALORACLE}(i,t)$ with input $n$ and $t$ which further calls the procedure $\textup{EVALORACLE}(i,t)$ at most $2$ times with input $n+1$ and $t$ or $n + 1$ and $t+p_n$. This gives the base case. Suppose that the statement is true for $i = n, n-1,\ldots,j + 1$. Now consider $i = j$ and $g_j(t)$ for any $t$, the procedure $\textup{EVALORACLE}(i,t)$ with input $j$ and $t$ further calls the procedure $\textup{EVALORACLE}(i,t)$ at most $2$ times with input $j+1$ and $t$ or $j+1$ and $t+p_i$. The former is equivalent to the value of $g_{j+1}(t)$ and the latter is equivalent to the value of $g_{j+1}(t+p_i)$. The inductive step shows that both call requires at most $2^{n-j}$ further recursive calls to the procedure $\textup{EVALORACLE}(i,t)$. Therefore, $g_i(t)$ for any $t$ requires a total of $2 \cdot 2^{n-j} = 2^{n-j+1}$ recursive calls, which completes the proof.  
\end{proof}
\begin{algorithm}[H]
\caption{Evaluation of $g_i(t)$ and optimal assignment $S$} \label{DPpricing}
\begin{algorithmic}[1]
\State{\textbf{input:} $i$, $t$}
\Procedure{EvalOracle}{$i$, $t$}
\If {$i = n + 1$} \Return $0$, $S = \emptyset$
\Else
	\If{$t \le a_i$} 
		\State \Return{\Call{EvalOracle}{$i$, $a_i$}}
	\ElsIf{$t \ge a_i + \pi_i$} 
		\State \Return{\Call{EvalOracle}{$i+1$, $t$}}, $S \gets S \cup \{0\}$
	\Else
		\If{$(a_i + \pi_i - t) + $\Call{EvalOracle}{$i+1$, $t + p_i$} $\ge$ \Call{EvalOracle}{$i+1$, $t$}}
	 		\State \Return{$(a_i + \pi_i - t) + $\Call{EvalOracle}{$i+1$, $t + p_i$}}, $S \gets S \cup \{1\}$
	 	\Else	 		
	 		\State \Return{\Call{EvalOracle}{$i+1$, $t$}}, $S \gets S \cup \{0\}$
		\EndIf
	\EndIf
\EndIf
\EndProcedure
\State \Return {\Call{EvalOracle}{$i,t$}, $S$}
\end{algorithmic}
\end{algorithm}

\subsubsection{The case of integral input data.} \label{ncapprox}
We now propose an implementation of the dynamic programming algorithm with a running time of $O(nc)$ in the case of integral input data. In Appendix B, we analyzed the case when the data is rational and we constructed the functions $g_i(t)$ in the interval $[0,c]$ by recursively evaluating the functions at the potential breakpoints from the set
\begin{align*}
\{0,e,2e,3e,\ldots,c\} \text{ where }e \in \{1/d,1/2d,1/3d,\ldots,1/(i+1)d\},
\end{align*}
where $d$ is the common denominator. If we now assume the input data are integral, we can further reduce the set to $\{0,1,2,\ldots,c\}$. Formally, we have the following theorem.
\begin{theorem} \label{alterdpthm}
Assume all input data are integral, it suffices to evaluate a function $g_i(t)$ at $t \in \{0,1,2,\ldots,c\}$ to construct the function. 
\end{theorem}
\begin{proof}[proof of Thoerem \ref{alterdpthm}.]
Consider any set of accepted flights, $S$, and let the flights in $A$ be $1, 2, \ldots, |A|$. Following similar arguments in the subsection \ref{approximation}, we can assign the park time of the first flight  to be $a_1$ and the gate becomes available for the second flight at time $\min\{a_2, a_1 + p_1\}$. Therefore, we see that the park time for flight $i$ is given by $\sum_{j=2}^{i-1} \min\{a_j, a_{j-1} + p_{j-1}\} + \sum_{j=2}^{i-1}p_j$ for $i \ge 3$. Note that the park time of flight $i$ is the sum of the park time and processing time of flight $i-1$. Since we assume all input data are integral, the park times for all flights are integral.

Next, we optimally assign the flights $1,2,\ldots,n$ in the set $\mathcal{F}'$ based on the recursive formula (\ref{recureq}). For any flight $j$, if we accept $j$, it can only be accepted at an integral time that is either determined by the accepted flights before $j$ or $a_j$. If we do not accept $j$, we move on to optimally assigning the flights $j+1, j+2, \ldots, n$ and the arguments for $j$ can be applied. Therefore, decisions on whether to accept the flights all occur at integral times and we only need to evaluate $g_i(t)$ at the integral times in $[0,c]$ to construct that $g_i(t)$.  
\end{proof}

Based on the above observation, an implementation of the dynamic programming algorithm with a running time of $O(nc)$ is shown below as Algorithm \ref{alterapprox}. Note that this is an implementation of the backward version of the dynamic programming algorithm. 
\begin{algorithm}[H]
\caption{Alternative implementation of the dynamic programming algorithm} \label{alterapprox}
\begin{algorithmic}[1]
\State Let $g_i(0 \ldots c)$ be a table where $i \in \{1,2,\cdots, n+1\}$. 
\For{$ t \gets c$ to $0$ } {$g_{n+1}(t) \gets 0$, $S_t \gets \emptyset$}
\EndFor
\For{$ i \gets n$ to $1$}
\For{$ t \gets c$ to $0$}
\If{$t \le a_i$} 
\State {$g_i(t) \gets g_i(a_i)$, $S_t \gets S_{a_i}$}
\ElsIf {$t \ge a_i + \pi_i$} 
\State {$g_i(t) \gets g_{i+1}(t)$, $S_t \gets S_t \cup \{0\}$}
\Else
\If {$a_i + \pi_i - t + g_{i+1}(t+p_i) \ge g_{i+1}(t)$ }
\State {$g_i(t) \gets a_i + \pi_i - t + g_{i+1}(t+p_i)$, $S_t \gets S_t \cup \{1\}$}
\Else 
\State {$g_i(t) \gets g_{i+1}(t)$, $S_t \gets S_t \cup \{0\}$}
\EndIf
\EndIf
\EndFor
\EndFor
\State \Return{$g_1(0)$, $S_0$}
\end{algorithmic}
\end{algorithm}
Lastly note that in the cases where not all input data are integers, this algorithm works as an approximation algorithm and provides an alternative to the submodular maximization approximation algorithm. This new approximation algorithm is referred to as the approximative dynamic programming algorithm (ADP) in the computational experiments.  
\subsection{Alternative reformulation of the pricing problems}
Alternatively, the pricing problem can be formulated as a variant of the shortest path problem with time windows known as the shortest path problem with time windows and time costs (SPPTWTC) in which there is an additional linear cost associated with the service start time at each node. Such a problem is studied in details in \cite{ESPPTWlinearnodecost}. To formulate the pricing problems as a SPPTWTC, a source and a sink need to be introduced and each of the flights in the set $\mathcal{F}'$ represents a node in the network. The cost associated with each arc between nodes $i$ and $j$ is given by the negative of the corresponding dual variable value, $-\pi_{ij}$. To solve the SPPTWTC problem, a dynamic programming algorithm is proposed in \cite{ESPPTWlinearnodecost}. This dynamic programming algorithm is derived based on the general labelling algorithm for the shortest path problems with time windows. For a detailed discussion of the labeling algorithm and some relevant extensions, we refer the readers to \cite{ESPPTWlabeling} and \cite{ESPPTWconvexcosts}. One of the key reasons for the effectiveness of this dynamic programming algorithm is the use of upper bounds on the service start times to eliminate labels that are infeasible with respect to the time windows, so that the number of labels created is significantly reduced. However, if we reformulate our pricing problems as SPPTWTC, the lack of upper bounds on the park times leads to an exponential number of labels rendering the dynamic programming algorithm proposed in \cite{ESPPTWlinearnodecost} ineffective. Nonetheless, if one wanted to enforce upper bounds on the park times, one would definitely be able to leverage the work in \cite{ESPPTWlabeling}, \cite{ESPPTWconvexcosts}, and \cite{ESPPTWlinearnodecost}. Therefore, we did not pursue this reformulation further. 
\subsection{Large-sized instances} \label{largesizedinstance}
The total number of arrivals into a busy international airport can be very large and the proposed dynamic programming algorithm can be slow to obtain the optimal assignments. We utilize two ways to further decompose the pricing problem of each gate to tackle large-sized instances. One is a $2$-approximation algorithm and the other is a standard rolling horizon method. As the standard rolling horizon method is common in many applications, we defer the details of implementations to Appendix C.
\subsubsection{Block decomposition approximation.} \label{largeapprox}
When dealing with a large number of arrival flights, the flights can be usually divided into blocks based on the interactions between the flights to reduce the problem sizes. We introduce the idea of an adjacency parameter, denoted by $\overline{\sigma}$. A formal definition of $\overline{\sigma}$ is as follows. 
\begin{definition} \label{adjacencyparameter}
For each flight $i$, let $j$ be the earliest flight after $i$ such that $a_j > a_i + \pi_i + p_i$, and denote $\sigma_i: = \min\{j-i, n-i\}$, then $\overline{\sigma} := \max_{1 \le i \le n} \sigma_i$. 
\end{definition}

Note that this adjacency parameter is likely to differ across the gates and change after a new iteration of the master problem is solved with the updated set $P_k$. The use of the adjacency parameter also appears in the work of \cite{delftbuitendijk}. Now we are ready to give the block decomposition approximation algorithm. 

\begin{algorithm}[H]
\caption{Block decomposition approximation} \label{rollingapproximation}
\begin{algorithmic}[1]
\State \textbf{Input:} adjacency parameter $\overline{\sigma}$
\State Divide the set $\mathcal{F}'$ into $\lfloor n/\overline{\sigma} \rfloor + 1$ blocks such that each of block has $\overline{\sigma}$ flights and the last block has $n - \lfloor n/\overline{\sigma} \rfloor \cdot \overline{\sigma}$ flights. 
\State Solve for the optimal assignment in each of the blocks
\State Compute the sum of the total net benefits of the odd-indexed blocks and even-indexed blocks and let $B$ be the collection of blocks with a larger sum of the total net benefits 
\State Construct an assignment, $S$, by cascading the optimal assignments from the blocks in $B$ 
\State Compute the total net benefits of the assignment $S$, $f(S)$
\State \Return{$S$, $f(S)$}
\end{algorithmic}
\end{algorithm}
Note that we refer to the discussions in Subsection \ref{approximation} for the computation of the total net benefits $f(S)$. We now provide an approximation guarantee for Algorithm \ref{rollingapproximation}.
\begin{theorem} \label{blockguarantee}
Let $S^*$ and OPT be the true optimal assignment and optimal objective value of the pricing problem respectively and let $S$ be the assignment returned by Algorithm \ref{rollingapproximation} and $OPT_S$ be the objective value associated with $S$, then $OPT_S \ge OPT/2$. 	
\end{theorem}
\begin{proof}[proof of Theorem \ref{blockguarantee}.]
Label all the blocks by $1,2,\ldots, \lfloor n/\overline{\sigma} \rfloor, \lfloor n/\overline{\sigma} \rfloor+1$ and let the optimal objective value of each block be $OPT_i$ for $ i \in [\lfloor n/\overline{\sigma} \rfloor+1]$. By the construction of the blocks, we have that 
\begin{equation}
	\sum_{i \in [\lfloor n/\overline{\sigma} \rfloor+1]} OPT_i \ge OPT.
\end{equation}

Now without loss of generality, we assume that the sum of the total net benefits of the odd-indexed blocks is larger than that of the even-indexed blocks, then we have that
\begin{equation}
	2\sum_{i \text{ is odd}} OPT_i \ge \sum_{i \text{ is odd}} OPT_i + \sum_{i \text{ is even}} OPT_i \ge OPT.
\end{equation}

Now we show that we can construct a valid assignment by just cascading the optimal assignments from the odd-indexed blocks. Let $I_1$, $I_2$ be any two consecutive odd-indexed blocks and let $f_{1}$ be the last flight in block $I_1$ and $f_2$ be the first flight in $I_2$ respectively. From the construction of the blocks, we have that $f_2 - f_1 > \overline{\sigma}$ since they are from two different blocks. In addition, by the definition of the adjacency parameter, we have $a_{f_2} > a_{f_1} + \pi_{f_1} + p_{f_1}$ and, in fact, $a_{f_2} > a_{i} + \pi_{i} + p_{i}$ for any flight $i$ in block $I_1$. Recall that $a_i + \pi_i$ is the end point of flight $i$'s acceptance window beyond which $i$ is not accepted. Therefore, the optimal assignment of the flights in the block $I_1$ does not conflict with whether to accept $f_2$ in the optimal assignment of the flights in the block $I_2$ and cascading the optimal assignments of $I_1$ and $I_2$ forms a valid assignment. Repeating this argument for all the odd-indexed blocks, we see cascading the optimal assignments of those blocks form a valid assignment pattern.

Therefore, we have that
\begin{equation}
OPT_S = \sum_{i \text{ is odd}} OPT_i \ge OPT/2,	
\end{equation} 
which completes the proof. 
\end{proof}
Note that Theorem \ref{blockguarantee} implies that if there exists an assignment pattern with positive total net benefits, the block decomposition approximation algorithm has to return an assignment pattern, $S$, with a positive $OPT_S$ and thus $S$ is a favorable assignment pattern. Furthermore, Algorithm \ref{rollingapproximation} can be improved as follows. Suppose the odd-indexed blocks have a larger sum of the total net benefits and we construct a valid assignment pattern from the odd-indexed blocks. We can still potentially accept more flights from the even-indexed blocks to contribute more net benefits as long as accepting those flights does not violate constraints (\ref{C1}) - (\ref{C3}).\\
Lastly, as the adjacency parameter, $\overline{\sigma}$, is computed based on the interactions between flights, it is not likely to be very large. Most flights park at the gate to get ready for the next flight leg with a relatively short turn time and thus they have limited interactions with other flights that arrive hours later.

\section{Feasible solutions and branching scheme} \label{branching}
\subsection{Feasible solutions}
After optimality is achieved in the linear programming relaxation of the master problem (\ref{colGenObj}) - (\ref{bound}), we solve the master problem with the binary requirements of the decision variables reinstated to obtain a feasible solution. We refer to this problem as the binary program. The objective value of the linear programming relaxation provides a lower bound on the optimal objective value while the binary program provides an upper bound on the optimal objective value. Subsequently, we can study the quality of the binary program solution by the following,
\begin{equation} \label{qualitysol}
gap_{root} = \frac{UB_{root} - LB_{root}}{LB_{root}},	
\end{equation}
where $UB_{root}$ is the upper bound from the binary program and $LB_{root}$ is the lower bound from the linear programming relaxation. In the cases where $LB_{root}$ is zero, we use the absolute gap of $UB_{root}-LB_{root}$ which is denoted by an additional notation of ``(a)" after the numericals. We will use the abbreviations of UB and LB to represent the upper bound from the binary program and lower bound from the linear programming relaxation respectively in the computational experiments section. 

\subsection{Branching on the assignment decisions} 
In many cases, although we can obtain a feasible solution from the binary program, the quality of that solution is poor. Thus, we propose a branching rule on the assignment decisions to obtain better solutions. Formally, let $z^*$ be an optimal solution to the linear programming relaxation of the master problem (\ref{colGenObj}) - (\ref{bound}).  The corresponding assignment decision for each flight $i$ given by 
\begin{equation} \label{optimalX}
y_{ik}^{*} : = \sum_{p \in P_k} \delta_{ip}^k z_p^{k*}, \quad \forall i \in \mathcal{F}, \; k \in \mathcal{G}.
\end{equation}

Consider a fractional $z^*$, as $\delta_{ip}^k$ is either $0$ or $1$,  there may exist $0 < y_{ik}^* < 1$ for a particular $i$ and $k$. When such a fractional assignment decision occurs, we denote the $i$ and $k$ values by $\hat{i}$ and $\hat{k}$. We present a branching scheme on this assignment decision which is adapted from \cite{Ryan}. We force $y_{\hat{i}\hat{k}} = 1$ on the left branch and $y_{\hat{i}\hat{k}} = 0$ on the right branch by adding valid inequalities. In particular, flight $\hat{i}$ can be forced to use gate $\hat{k}$ by the following inequality,
\begin{equation}
\sum_{p \in P_{\hat{k}}} (1 - \delta_{\hat{i}p}^{\hat{k}}) z_p^{\hat{k}}+ \sum_{k \in \mathcal{F} \backslash \hat{k}} \sum_{p \in P_k} \delta_{\hat{i}p}^k z_p^k \le 0.
\end{equation}

On the other side of the branch where flight $\hat{i}$ can be forced not to use gate $\hat{k}$ by the following inequality,
\begin{equation}
\sum_{p \in P_{\hat{k}}} \delta_{\hat{i}p}^{\hat{k}} z_p^{\hat{k}}\le 0.
\end{equation}

Note that after the branching constraints are added, the objective functions in the pricing problems have to be updated to incorporate the dual informations of these constraints. By updating the dual variables of the constraints (\ref{coverConstraint}) and (\ref{availConstraint}), we can keep the objective functions in the same format. Let the dual variables to the branching constraints be $\lambda_{\hat{i}\hat{k}}$. On the left branches, if $k = \hat{k}$, we have the new objective function as
\begin{align}
\min \; & \sum_{i \in \mathcal{F}} (t_i^{g}-a_i) - \sum_{i \in \mathcal{F}} x_{ik}\pi_i - \mu_k - (1 - x_{\hat{i} k}) \lambda_{\hat{i}\hat{k}} \notag \\
\Leftrightarrow
\min \; & \sum_{i \in \mathcal{F}} (t_i^{g}-a_i) - \left(\sum_{i \in \mathcal{F} \backslash \hat{i}} x_{ik}\pi_i + x_{\hat{i} k} (\pi_{\hat{i}} + \lambda_{\hat{i}\hat{k}})\right) - (\mu_k + \lambda_{\hat{i}\hat{k}}).
\end{align}

When $k \neq \hat{k}$, the new objective function becomes
\begin{align} \label{newobjectbranching}
\min \; & \sum_{i \in \mathcal{F}} (t_i^{g}-a_i) - \sum_{i \in \mathcal{F}} x_{ik}\pi_i - \mu_k - x_{\hat{i} k} \lambda_{\hat{i}\hat{k}} \notag \\
\Leftrightarrow
\min \; & \sum_{i \in \mathcal{F}} (t_i^{g}-a_i) - \left(\sum_{i \in \mathcal{F} \backslash \hat{i}} x_{ik}\pi_i + x_{\hat{i} k}(\pi_{\hat{i}} + \lambda_{\hat{i}\hat{k}}) \right)- \mu_k.
\end{align}

Now on the right branches, if $k = \hat{k}$, we obtain the same objective function as shown in (\ref{newobjectbranching}). When $k \neq \hat{k}$, the objective function remains unchanged.

We want to point out that the dynamic programming and approximation algorithms have to be adapted at a node where some branching constraints have been added, otherwise the assignment patterns that are generated can potentially violate the branching constraints. In particular, it is easy to modify the pricing problem methods for the right branches. If flight $\hat{i}$ is forced not to use gate $\hat{k}$, we can remove flight $\hat{i}$ from the subset of flights $\mathcal{F}'$ during the pre-processing. However, it is more difficult on the left branches when we force flight $\hat{i}$ to use gate $\hat{k}$. As flight $\hat{i}$ can be accepted at any time in the interval, $[a_{\hat{i}}, a_{\hat{i}} + \pi_{\hat{i}}]$, we need to first determine the optimal time to accept flight $\hat{i}$ after which the rest of the flights in the set $\mathcal{F}'$ can be optimally assigned. This is more difficult to implement because the time in $[a_{\hat{i}}, a_{\hat{i}} + \pi_{\hat{i}}]$ takes continuous values. To address this issue, we restrict the branching options in the dynamic program based on assignment decisions and then solve the pricing problems to optimality using a standard solver. While it can be time consuming for larger instances to solve pricing problems directly, we observed in the experiments that branching is often not needed for those larger instances and consequently this modification does not impact the performances. On the other hand, for the smaller instances where branching is needed, using a solver does not slow down the whole procedure significantly and we can still obtain favorable results.

\section{Computational experiments} \label{numericalexp}
\subsection{Instance generation and initialization}
We test the proposed methods on randomly generated instances and instances derived from real world data. The detailed information about each instance is reported in Table \ref{instanceSize} and a summary of the real world operational data is presented in Table \ref{ATLinfo}. For the randomly generated instances, we generate the inter-arrival time of two consecutive flights based on a uniform distribution such that each gate has an arriving flight every $75$ to $180$ minutes on average depending on the size of the instances. We use the minimum turn time data from a major U.S. airline. Furthermore, we generate each gate with a type which determines whether it can accommodate heavy aircraft, a set of eligible airlines, and a buffer time. The buffer times are set to be identical in our experiments for simplicity, but they can be adjusted accordingly if needed. The compatibilities data, $\alpha_{ik}$, used in constraints (\ref{MIPC6}) and (\ref{compatibility}) as well as pre-processing step of the pricing problem algorithms, can be computed based on the aircraft type, the gate type, and the set of eligible airlines.

For the real world instances, we use arrivals from multiple days in 2019 before the COVID-19 pandemic from the Atlanta Hartfield-Jackson Airport, the busiest airport by passenger traffic. The data are obtained from the U.S. Bureau of Transportation Statistics website. Note that for instances that do not use arrivals in the whole 24 hours' period, we report the time interval during which the arrivals are used. As it is difficult to obtain the precise information about the gates, we generate the gates in a similar way as described above, but with an additional consideration. As very large percentage of arrivals into ATL are Delta Air Lines (DAL) flights and DAL operates many gates at ATL, a large number of gates allow DAL flights.

Once we have the set of arrivals and the set of gates for an instance, we initialize the instance. For each of the test instance, the set of feasible patterns $P_k$ for all gate $k$ have to be initialized to contain at least one pattern to satisfy the availability constraint (\ref{availConstraint}) in the master problem. Moreover, the union of $P_k$ has to satisfy the covering constraint (\ref{coverConstraint}). Since we start with empty $P_k$, we provide each gate a feasible assignment pattern by randomly assigning each flight to a compatible gate to have an overall feasible assignment. Consequently, at the beginning of the proposed procedure, there is an assignment pattern for each gate. We note that although this is a simple initialization procedure, we do not expect any more complex initialization procedure could show significant improvements. In addition, our termination criteria is a relative gap of $2\%$ or an absolute gap of $0.5$ (in case the optimal objective of the root node LP relaxation is $0$.

\begin{longtable}{ccccccc}
\caption{Instance information.\label{instanceSize}}\\
\hline
no. (name) & size & source & no. of flights, gates & flight/gate & inter-arr(min.)\\
\hline\hline
1 (f30g5s) & small & synthetic & 30 flights, 5 gates & $6.00$ & $12.07$\\
2 (f30g10s) & small & synthetic & 30 flights, 10 gates & $3.00$ & $9.66$\\
3 (f50g10s) & small & synthetic & 50 flights, 10 gates & $5.00$ & $8.53$\\
4 (f50f20s) & small & synthetic & 50 flights, 20 gates & $2.50$ & $4.37$\\
5 (f100g35s) & small & synthetic & 100 flights, 35 gates & $2.86$ & $2.72$\\
6 (f100g50s) & small & synthetic & 100 flights, 50 gates & $2.00$ & $1.75$\\\hline\hline
7 (f150g50s1) & moderate & synthetic & 150 flights, 50 gates & $3.00$ & $2.01$\\ 
8 (f150g50s2) & moderate & synthetic & 150 flights, 50 gates & $3.00$ & $2.30$\\
9 (f150g50s3) & moderate & synthetic & 150 flights, 50 gates & $3.00$ & $2.10$\\\hline
10 (f200g100a1) & moderate & 10:00-13:12, 08/23/19 & 200 flights, 100 gates & $2.00$ & $0.96$\\
11 (f200g100a2) & moderate & 13:00-15:51, 08/23/19 & 200 flights, 100 gates & $2.00$ & $0.86$\\\hline
12 (f300g150s1) & moderate & synthetic & 300 flights, 150 gates & $2.00$ & $0.51$\\ 
13 (f300g150s2) & moderate & synthetic & 300 flights, 150 gates & $2.00$ & $0.49$\\
14 (f300g150s3) & moderate & synthetic & 300 flights, 150 gates & $2.00$ & $0.48$\\\hline
15 (f300g150a1) & moderate & 10:00-14:31, 08/23/19 & 300 flights, 150 gates & $2.00$ & $0.90$\\
16 (f300g150a2) & moderate & 13:00-17:47, 08/23/19 & 300 flights, 150 gates & $2.00$ & $0.95$\\\hline\hline
17 (f800g200s1) & large & synthetic & 800 flights, 200 gates & $4.00$ & $0.77$\\
18 (f800g200s2) & large & synthetic & 800 flights, 200 gates & $4.00$ & $0.73$\\
19 (f800g200s3) & large & synthetic & 800 flights, 200 gates & $4.00$ & $0.75$\\\hline
20 (f800g200a1) & large & 10:00-23:13, 08/19/19 & 800 flights, 200 gates & $4.00$ & $0.99$\\
21 (f800g200a2) & large & 10:00-23:13, 08/20/19 & 800 flights, 200 gates & $4.00$ & $0.99$\\
22 (f800g200a3) & large & 10:00-23:41, 08/21/19 & 800 flights, 200 gates & $4.00$ & $1.02$\\
23 (f800g200a4) & large & 10:00-22:36, 08/22/19 & 800 flights, 200 gates & $4.00$ & $0.95$\\
24 (f800g200a5) & large & 10:00-22:22, 08/23/19 & 800 flights, 200 gates & $4.00$ & $0.93$\\\hline
25 (f1000g200s1) & large & synthetic & 1000 flights, 200 gates & $5.00$ & $0.93$\\
26 (f1000g200s2) & large & synthetic & 1000 flights, 200 gates & $5.00$ & $0.89$\\
27 (f1000g200s3) & large & synthetic & 1000 flights, 200 gates & $5.00$ & $0.89$\\\hline
28 (f1000g200a1) & large & 6:00-21:50, 08/19/19 & 1000 flights, 200 gates & $5.00$ & $0.95$\\
29 (f1000g200a2) & large & 6:00-21:41, 08/20/19 & 1000 flights, 200 gates & $5.00$ & $0.94$\\
30 (f1000g200a3) & large & 6:00-21:37, 08/21/19 & 1000 flights, 200 gates & $5.00$ & $0.94$\\
31 (f1000g200a4) & large & 6:00-21:21, 08/22/19 & 1000 flights, 200 gates & $5.00$ & $0.92$\\
32 (f1000g200a5) & large & 6:00-21:21, 08/23/19 & 1000 flights, 200 gates & $5.00$ & $0.92$\\
\hline
33 (f1113g192a19) & large & 08/19/19 & 1113 flights, 192 gates & $5.80$ & $1.28$\\
34 (f1095g192a20) & large & 08/20/19 & 1095 flights, 192 gates & $5.70$ & $1.30$\\
35 (f1092g192a21) & large & 08/21/19 & 1092 flights, 192 gates & $5.69$ & $1.31$\\
36 (f1125g192a22) & large & 08/22/19 & 1125 flights, 192 gates & $5.86$ & $1.28$\\
37 (f1125g192a23) & large & 08/23/19 & 1125 flights, 192 gates & $5.86$ & $1.27$\\\hline
\end{longtable}

\begin{longtable}{ccc}
\caption{Summary of arrivals at ATL.\label{ATLinfo}}\\
\hline 
date. & no. of gates & no. of arrivals \\
\hline\hline
Aug 19, 2019 & $192$ & $1113$ ($687$ Delta, $39$ American, $8$ United, $379$ Other)\\
Aug 20, 2019 & $192$ & $1095$ ($681$ Delta, $32$ American, $9$ United, $373$ Other)\\
Aug 21, 2019 & $192$ & $1092$ ($675$ Delta, $35$ American, $10$ United, $372$ Other)\\
Aug 22, 2019 & $192$ & $1125$ ($684$ Delta, $40$ American, $11$ United, $390$ Other)\\
Aug 23, 2019 & $192$ & $1125$ ($684$ Delta, $39$ American, $11$ United, $391$ Other)\\\hline
\end{longtable}
\subsection{Software and hardware}
For the implementation, the pricing problem algorithms are implemented in Python. Gurobi (version 9.1) is used whenever a commercial solver is needed. The computations are performed on an Unix system with a $12$-core CPU and $16$GB RAM. We also set a limit of $8$ hours for each instance.
\subsection{Computational results}
For a complete comparison, the performance of the compact mixed integer programming (MIP) formulation (\ref{MIPO1}) - (\ref{MIPcontin}) is considered using the small-sized instances and shown in Table \ref{Mip}. The formulation is solved using the Gurobi solver. As we have noted before, the MIP formulation is not ideal for this problem as evidenced in Table \ref{Mip}.\\

\begin{longtable}{ccccc}
\caption{Performances of the MIP formulation.\label{Mip}}\\
\hline
method & instance no. & time(sec.) & incumbent obj. & best bound\\
\hline\hline
\multirow{6}{5em}{\centering{MIP}} & 1 (f30g5s) & 339.78 & 1203.26 & 1203.26\\
& 2 (f30g10s) & 6948.33 & 258.53 & 258.53\\ 
& 3 (f50g10s) & 5347.46 & 74.73 & 74.73\\ 
& 4 (f50f20s) & incomplete & 18.64 & 0.00\\
& 5 (f100g35s) & incomplete & 1.66 & 0.00\\
& 6 (f100g50s) & incomplete & 1.14 & 0.00\\\hline
\end{longtable}
For the following results, note that we proposed to obtain a feasible solution at the root node by reinstating the binary requirement and compute the quality of the solution by (\ref{qualitysol}). In addition, we use the following notations to report results in tables:
\begin{itemize}
	\item ct: computation time in second;
	\item rt: computation time spent on the root node in second;
	\item node(s): number of nodes in the branch-and-bound tree searched to obtain the reported solution, where ``1" represents the root node.
\end{itemize}
Additionally, we use the notations ``$LB_{root}$", ``$UB_{root}$", and ``$gap_{root}$" introduced previously. Note that these values are obtained at the root node.
\subsubsection{Small-sized instances.}
It is important to note that we can combine the pricing problem algorithms presented in Section \ref{solvingpricing} in our implementation. For the small-sized instances, we test the following ways: 
\begin{enumerate}
	\item Gurobi solver (S)
	\item Dynamic programming algorithm (DP)
	\item $70$ iterations of submodular maximization followed by the dynamic programming algorithm (SM+DP)
	\item $25$ iterations of approximative dynamic programming algorithm followed by the dynamic programming algorithm (ADP+DP).
\end{enumerate}

We run a fixed number of iterations of both SM and ADP across different instances as discussed above, and no attempt has been made to tune these parameters. For a particular class of instances, tuning these parameters may provide even better results.

In the solver option (Option 1 above), specifically, we set the Gurobi PoolSearchMode parameter to $2$ and request the solver to provide up to $75$ feasible solutions and add as many as possible to the  master problem. The updated solver parameters does not deteriorate the performance of the solver as we observe the solver very rarely performs extra computations after the optimal solution is found to fill the feasible solution pool. These extra feasible solutions are likely assignment patterns with favorable reduced costs and can potentially reduce the number of times the pricing problems are solved.

The detailed computational results for the small-sized instances are shown in Table \ref{pricingCombinations}. We see that our proposed approaches out-perform the Gurobi solver (S) in all six instances, even with just the dynamic programming algorithm alone (DP). While the approximation algorithms are designed to provide additional improvements relative to the dynamic programming algorithm, however, they do not offer any substantial boost in these small instances and, on the contrary, we see the dynamic programming algorithm alone achieves better computation time. In addition, the binary program solutions obtained with different methods differ in some of instances. This is likely due to three methods, namely, DP, SM, and ADP, are generating different patterns. In general, if we increase the size of the instance, it becomes more difficult to solve. We observe that the flight to gate ratio has an impact on the time. In particular, although instance $1$ is smaller in size than instances $2$ and $4$ are, it takes longer to solve than both instances $2$ and $4$ do. The flight to gate ratio determines how congested the gates are. The higher the ratio, the larger the number of flights each gate on average has to accept. Instances with higher ratio require more delicate assignments. Furthermore, in contrast to the worse case theoretical analysis, the dynamic programming algorithm performs well. This is likely due to the fact that the inter-arrival times are much shorter than the processing times and consequently many recursive calls to the evaluation oracle have input $t$ beyond the corresponding flights' acceptance windows. \\
\begin{longtable}{cccccccccc}
\caption{Pricing problem methods on small-sized instances.\label{pricingCombinations}}\\
\hline
instance & methods & ct(sec.) & rt(sec.) & $LB_{root}$ & $UB_{root}$ & $gap_{root}$($\%$) & final obj. & node(s) \\
\hline
\hline
\multirow{4}{6em}{\centering{1\\(f30g5s)}} & S & 130.71 & 130.71 & 1203.26 & 1203.26 & 0.00 & 1203.26 & 1\\
& DP & 68.83 & 68.83 & 1203.26 & 1203.26 & 0.00 & 1203.26 & 1\\
& SM+DP & 48.46 & 27.52 & 1203.26 & 1448.98 & 20.4 & 1203.26 & 3\\
& ADP+DP & 33.64 & 33.64 & 1203.26 & 1203.26 & 0.00 & 1203.26 & 1\\\hline
\multirow{4}{6em}{\centering{2\\(f30g10s)}} & S & 26.52 & 26.52 & 258.53 & 258.53 & 0.00 & 258.53 & 1\\
& DP & 2.15 & 2.15 & 258.53 & 258.53 & 0.00 & 258.53 & 1\\
& SM+DP & 0.63 & 0.63 & 258.53 & 258.53 & 0.00 & 258.53 & 1\\
& ADP+DP & 3.52 & 3.52 & 258.53 & 258.53 & 0.00 & 258.53 & 1\\\hline
\multirow{4}{6em}{\centering{3\\(f50g10s)}} & S & 91.85 & 91.85 & 74.73 & 74.73 & 0.00 & 74.73 & 1\\
& DP & 39.86 & 25.52 & 74.73 & 77.84 & 4.00 & 74.73 & 5\\
& SM+DP & 38.60 & 26.73 & 74.73 & 80.79 & 7.50 & 74.73 & 5\\
& ADP+DP & 30.79 & 30.79 & 74.73 & 74.73 & 0.00 & 74.73 & 1\\\hline
\multirow{4}{6em}{\centering{4\\(f50f20s)}} & S & 47.63 & 47.63 & 18.64 & 18.64 & 0.00 & 18.64 & 1\\
& DP & 1.72 & 1.72 & 18.64 & 18.64 & 0.00 & 18.64 & 1\\
& SM+DP & 1.44 & 1.44 & 18.64 & 18.64 & 0.00 & 18.64 & 1\\
& ADP+DP & 10.07 & 10.07 & 18.64 & 18.64 & 0.00 & 18.64 & 1\\\hline
\multirow{4}{6em}{\centering{5\\(f100g35s)}} & S & 1495.42 & 1495.42 & 1.66 & 1.66 & 0.00 & 1.66 & 1\\
& DP & 189.61 & 37.92 & 1.66 & 1.88 & 11.70 & 1.66 & 26\\
& SM+DP & 20.67 & 20.67 & 1.66 & 1.66 & 0.00 & 1.66 & 1\\
& ADP+DP & 46.89 & 46.89 & 1.66 & 1.66 & 0.00 & 1.66 & 1\\\hline
\multirow{4}{6em}{\centering{6\\(f100g50s)}} & S & 1893.26 & 1893.26 & 1.14 & 1.14 & 0.00 & 1.14 & 1\\
& DP & 17.89 & 17.89 & 1.14 & 1.14 & 0.00 & 1.14 & 1\\
& SM+DP & 10.06 & 10.06 & 1.14 & 1.14 & 0.00 & 1.14 & 1\\
& ADP+DP & 29.67 & 29.67 & 1.14 & 1.14 & 0.00 & 1.14 & 1\\\hline
\end{longtable}

\subsubsection{Moderate-sized instances.}
Next we show the computational results for the moderate-sized instances. As we have seen that the solver option is not as effective as other options are in solving the small-sized instances, only the options of DP, SM+DP, and ADP+DP are considered in this set of experiments. As we pointed out previously, we run $70$ iterations of SM and $25$ iterations of ADP across all instances.

The results are shown in Table \ref{moderateInstance}.  We see the same trend of longer time taken to solve an instance of larger size. Since the flight to gate ratios are small for these moderate-sized instances, we do not see any obvious impact of the ratio. The boost in performances from the submodular maximization approximation algorithm is much more obvious in the moderated-sized instances and in some cases, the approximative dynamic programming algorithm improves the computation time. The running time of either approximation algorithm increases linearly with the size of the set of flights. Given a larger set of flights, each iteration of the approximation algorithm takes much less time compared to each iteration of the dynamic programming algorithm. Nonetheless, we observe that the submodular maximization out-performs the approximative dynamic programming in all instances as the assignment patterns that are generated by the submodular maximization are usually of better total net benefits than those that are generated by the approximative dynamic programming. Moreover, we observe small variations in the times taken across instances of the same size and this suggests that besides the sizes of the set of flights and the set of gates, other factors associated with the flights and the gates can complicate the problem. In addition, although we see that even though some of  the synthetic instances and real world instances have the same number of flights and gates, the performances are very different and it is likely due to arrivals in the real world instances are much more complex than those in the synthetic instances.\\

\begin{longtable}{ccccccccc} 
\caption{Pricing problem methods on moderate-sized instances.\label{moderateInstance}}\\	
\hline
instance & methods & ct(sec.) & rt(sec.) & $LB_{root}$ & $UB_{root}$ & $gap_{root}$($\%$) & final obj. & node(s) \\
\hline
\hline
\multirow{3}{6em}{\centering{7\\(f150g50s1)}} & DP & 770.74 & 479.53 & 0.0 & 0.69 & 0.69(a) & 0.0 & 82\\
& SM+DP & 346.58 & 346.58 & 0.0 & 0.0 & 0.00(a) & 0.0 & 1\\
& ADP+DP & 985.58 & 985.58 & 0.0 & 0.22 & 0.22(a) & 0.22 & 1\\\hline
\multirow{3}{6em}{\centering{8\\(f150g50s2)}} & DP & 2108.45 & 1866.29 & 0.90 & 1.11 & 23.3 & 0.9 & 23\\
& SM+DP & 486.85 & 156.04 & 0.90 & 5.32 & 491 & 0.9 & 50\\
& ADP+DP & 1034.31 & 631.49 & 0.9 & 1.90 & 52.6 & 0.9 & 50\\\hline
\multirow{3}{6em}{\centering{9\\(f150g50s3)}} & DP & 1185.55 & 1185.55 & 10.74 & 10.74 & 0.00 & 10.74 & 1\\
& SM+DP & 439.70 & 439.70 & 10.74 & 10.74 & 0.00 & 10.74 & 1\\
& ADP+DP & 498.51 & 498.51 & 10.74 & 10.74 & 0.0 & 10.74 & 1\\\hline
\multirow{3}{6em}{\centering{10\\(f200g100a1)}} & DP & 5002.04 & 5002.04 & 0.0 & 0.0 & 0.00(a) & 0.0 & 1\\
& SM+DP & 929.45 & 929.45 & 0.0 & 0.0 & 0.00(a) & 0.0 & 1\\
& ADP+DP & 4689.39 & 4689.39 & 0.0 & 0.0 & 0.00(a) & 0.0 & 1\\\hline
\multirow{3}{6em}{\centering{11\\(f200g100a2)}} & DP & 2883.41 & 2883.41 & 0.0 & 0.0 & 0.00(a) & 0.0 & 1\\
& SM+DP & 570.75 & 570.75 & 0.0 & 0.0 & 0.00(a) & 0.0 & 1\\
& ADP+DP & 3168.33 & 3168.33 & 0.0 & 0.0 & 0.00(a) & 0.0 & 1\\\hline
\multirow{3}{6em}{\centering{12\\(f300g150s1)}} & DP & 8918.69 & 8918.69 & 349.10 & 349.10 & 0.00 & 349.10 & 1\\
& SM+DP & 2833.93 & 2833.93 & 349.10 & 349.10 & 0.00 & 349.10 & 1\\
& ADP+DP & 6891.40 & 6891.40 & 349.10 & 349.10 & 0.00 & 349.10 & 1 \\\hline
\multirow{3}{6em}{\centering{13\\(f300g150s2)}} & DP & 7927.31 & 7927.31 & 203.00 & 203.00 & 0.00 & 203.00 & 1\\
& SM+DP & 3111.97 & 3111.97 & 203.10 & 203.10 & 0.00 & 203.10 & 1\\
& ADP+DP & 6284.43 & 6284.43 & 203.10 & 203.10 & 0.00 & 203.10 & 1\\\hline
\multirow{3}{6em}{\centering{14\\(f300g150s3)}} & DP & 7830.18 & 7830.18 & 414.12 & 414.12 & 0.00 & 414.12 & 1\\
& SM+DP & 4570.69 & 4570.69 & 414.12 & 414.12 & 0.00 & 414.12 & 1\\
& ADP+DP & 8734.09 & 8734.09 & 414.12 & 414.12 & 0.00 & 414.12 & 1\\\hline
\multirow{3}{6em}{\centering{15\\(f300g150a1)}} & DP & 25132.32 & 25132.32 & 0.0 & 0.0 & 0.00(a) & 0.0 & 1\\
& SM+DP & 14032.56 & 14032.56 & 0.0 & 0.0 & 0.00(a) & 0.0 & 1\\
& ADP+DP & 23449.23 & 23449.23 & 0.0 & 0.0 & 0.00(a) & 0.0 & 1\\\hline
\multirow{3}{6em}{\centering{16\\(f300g150a2)}} & DP & 15249.75 & 15249.75 & 0.0 & 0.0 & 0.00(a) & 0.0 & 1\\
& SM+DP & 13052.81 & 13052.81 & 0.0 & 0.0 & 0.00(a) & 0.0 & 1\\
& ADP+DP & 19874.84 & 19874.84 & 0.0 & 0.0 & 0.00(a) & 0.0 & 1\\\hline
\end{longtable}

\subsubsection{Large-sized instances.}
Following the experiments on the moderate-sized instances, we move on to the large-sized instances. As we discussed in the Subsection \ref{largesizedinstance}, we can utilize the block decomposition approximation and rolling horizon framework to tackle the large-sized instances. The horizon size and window size are important parameters in the rolling horizon framework. We first perform parametric studies to understand how the horizon size and window size affect the performances. We vary the horizon size and window size to test the performances of the rolling horizon method on random large-sized instances. We observe that for any window size, smaller horizon sizes result in better performances. However, very small horizon size does not further reduce the time taken as the quality of the assignment patterns that are generated under small horizon sizes is poor and more iterations are needed to reach the optimality. 

For the large-sized instances, we again have a few possible ways to solve the pricing problems because we can either use a fixed or a dynamically determined horizon size for the rolling horizon method and also utilize the approximation algorithms. Here is a detailed breakdown: 
\begin{enumerate}
	\item Rolling horizon method (horizon size: $\overline{\sigma}$, window size: $1$) (RHD)
	\item Rolling horizon method (horizon size: $20$, window size: $1$) (RHF)
	\item Rolling horizon method (horizon size: $\min(20,\overline{\sigma})$, window size: $1$) (RHM)
	\item Submodular maximization if $\overline{\sigma} > 60$ in the first 25 iterations and rolling horizon method otherwise (horizon size: $\min(20,\overline{\sigma})$, window size: $1$) (SM+RHM)
	\item $30$ iterations of block decomposition approximation followed by the rolling horizon method (horizon size: $\overline{\sigma}$, window size: $1$) (BD+RHD)
\end{enumerate}
where $\overline{\sigma}$ is the adjacency parameter that depends on the values of the dual variables as described in Subsection \ref{largeapprox}. Furthermore, submodular maximization can generally provide assignment patterns of good quality and has a linear running time in the size of the set of flights, so its performance remains very effective for these large-sized instances. It can potentially improve the performances and serve as a benchmark for the block decomposition approximation. Note that whenever an algorithm is needed to evaluate the optimal assignment during the rolling horizon process or the block decomposition approximation, the direct implementation of the dynamic programming algorithm (Algorithm \ref{DPpricing}) is used. Again, for these large-sized instances, we use the same parameters across different instances and there can be potential improvements if parameters are tuned for individual cases. Nonetheless, we present results with a fixed set of parameters which still confirm the strong performances of our proposed approaches.

The results of the computations are shown in Table \ref{largeInstance}. After extensive experiments, we observe that the adjacency parameters, $\overline{\sigma}$, can be very large at beginning and, consequently, the options of RHD and BD+RHD that involve decompositions based on these parameters become very ineffective and unable to obtain a good solution within the allocated time limit. Therefore, we do not report the results of these two options in Table \ref{largeInstance}. From the results in the table, we see the rolling horizon method with a fixed horizon size performs much better than the same method with a horizon size of $\overline{\sigma}$ which struggles to solve these large-sized instances. If we further keep the horizon size at most $20$ in the RHM option, we obtain comparable if not slightly better performances in all instances compared to the rolling horizon method with a fixed horizon size.  However, submodular maximization algorithm efficiently provides assignment patterns of good quality when $\overline{\sigma}$ is large and improves the performances compared to both RHF and RHM options in all instances. The benefits of using submodular maximization algorithm is especially significant in the cases derived from ATL arrivals. The reductions in computation time can be as large as about $50\%$. For instances $33$-$37$, we show that our proposed approach can solve the flight-to-gate problem with arrivals in a single day for the busiest airport in the world within very reasonable amount of time.\\

\begin{longtable}{cccccccccc}
\caption{Pricing problem methods on large-sized instances.\label{largeInstance}}\\
\hline
instance & methods & ct(sec.) & rt(sec.) & $LB_{root}$ & $UB_{root}$ & $gap_{root}$($\%$) & final obj. & node(s)\\
\hline
\hline
\multirow{3}{6em}{\centering{17\\(f800g200s1)}} & RHF & 1160.33 & 1160.33 & 0.0 & 0.0 & 0.00(a) & 0.0 & 1\\
& RHM & 1147.62 & 1147.62 & 0.0 & 0.0 & 0.00(a) & 0.0 & 1\\
& SM+RHM & 1067.63 & 1067.63 & 0.0 & 0.0 & 0.00(a) & 0.0 & 1\\\hline
\multirow{3}{6em}{\centering{18\\(f800g200s2)}} & RHF & 1166.7 & 1166.7 & 0.0 & 0.0 & 0.00(a) & 0.0 & 1\\
& RHM & 1076.56 & 1076.56 & 0.0 & 0.0 & 0.00(a) & 0.0 & 1\\
& SM+RHM & 1010.11 & 1010.11 & 0.0 & 0.0 & 0.00(a) & 0.0 & 1\\\hline
\multirow{3}{6em}{\centering{19\\(f800g200s3)}} & RHF & 1015.80 & 1015.80 & 0.0 & 0.0 & 0.00(a) & 0.0 & 1\\
& RHM & 899.28 & 899.28 & 0.0 & 0.0 & 0.00(a) & 0.0 & 1\\
& SM+RHM & 975.88 & 975.88 & 0.0 & 0.0 & 0.00(a) & 0.0 & 1\\\hline
\multirow{3}{6em}{\centering{20\\(f800g200a1)}} & RHF & 2313.15 & 2313.15 & 0.0 & 0.0 & 0.00(a) & 0.0 & 1\\
& RHM & 2381.8 & 2381.8 & 0.0 & 0.0 & 0.00(a) & 0.0 & 1\\
& SM+RHM & 1570.65 & 1570.65 & 0.0 & 0.0 & 0.00(a) & 0.0 & 1\\\hline
\multirow{3}{6em}{\centering{21\\(f800g200a2)}} & RHF & 2033.91 & 2033.91 & 0.0 & 0.0 & 0.00(a) & 0.0 & 1\\
& RHM & 2138.42 & 2138.42 & 0.0 & 0.0 & 0.00(a) & 0.0 & 1\\
& SM+RHM & 1708.62 & 1708.62 & 0.0 & 0.0 & 0.00(a) & 0.0 & 1\\\hline
\multirow{3}{6em}{\centering{22\\(f800g200a3)}} & RHF & 2658.21 & 2658.21 & 0.0 & 0.0 & 0.00(a) & 0.0 & 1\\
& RHM & 2673.93 & 2673.93 & 0.0 & 0.0 & 0.00(a) & 0.0 & 1\\
& SM+RHM & 1714.31 & 1714.31 & 0.0 & 0.0 & 0.00(a) & 0.0 & 1\\\hline
\multirow{3}{6em}{\centering{23\\(f800g200a4)}} & RHF & 2457.84 & 2457.84 & 0.0 & 0.0 & 0.00(a) & 0.0 & 1\\
& RHM & 2444.43 & 2444.43 & 0.0 & 0.0 & 0.00(a) & 0.0 & 1\\
& SM+RHM & 1769.04 & 1769.04 & 0.0 & 0.0 & 0.00(a) & 0.0 & 1\\\hline
\multirow{3}{6em}{\centering{24\\(f800g200a5)}} & RHF & 30932.55 & 3093.55 & 0.0 & 0.0 & 0.00(a) & 0.0 & 1\\
& RHM & 2989.79 & 2989.79 & 0.0 & 0.0 & 0.00(a) & 0.0 & 1\\
& SM+RHM & 1818.73 & 1818.73 & 0.0 & 0.25 & 0.25(a) & 0.25 & 1\\\hline
\multirow{3}{6em}{\centering{25\\(f1000g200s1)}} & RHF & 1675.46 & 1675.46 & 0.0 & 0.0 & 0.00(a) & 0.0 & 1\\
& RHM & 1540.65 & 1540.65 & 0.0 & 0.0 & 0.00(a) & 0.0 & 1\\
& SM+RHM & 1440.66 & 1440.66 & 0.0 & 0.0 & 0.00(a) & 0.0 & 1\\\hline
\multirow{3}{6em}{\centering{26\\(f1000g200s2)}} & RHF & 2286.41 & 2286.41 & 0.0 & 0.0 & 0.00(a) & 0.0 & 1\\
& RHM & 1959.05 & 1959.05 & 0.0 & 0.0 & 0.00(a) & 0.0 & 1\\
& SM+RHM & 1843.98 & 1843.98 & 0.0 & 0.0 & 0.00(a) & 0.0 & 1\\\hline
\multirow{3}{6em}{\centering{27\\(f1000g200s3)}} & RHF & 1720.46 & 1720.46 & 0.0 & 0.0 & 0.00(a) & 0.0 & 1\\
& RHM & 1585.0 & 1585.0 & 0.0 & 0.0 & 0.00(a) & 0.0 & 1 \\
& SM+RHM & 1580.16 & 1580.16 & 0.0 & 0.0 & 0.00(a) & 0.0 & 1\\\hline
\multirow{3}{6em}{\centering{28\\(f1000g200a1)}} & RHF & 5858.02 & 5858.02 & 0.0 & 0.0 & 0.00(a) & 0.0 & 1\\
& RHM & 6354.51 & 6354.51 & 0.0 & 0.0 & 0.00(a) & 0.0 & 1\\
& SM+RHM & 3407.24 & 3407.24 & 0.0 & 0.0 & 0.00(a) & 0.0 & 1\\\hline
\multirow{3}{6em}{\centering{29\\(f1000g200a2)}} & RHF & 7270.22 & 7270.22 & 0.0 & 0.0 & 0.00(a) & 0.0 & 1\\
& RHM & 7342.51 & 7342.51 & 0.0 & 0.0 & 0.00(a) & 0.0 & 1\\
& SM+RHM & 3206.97 & 3206.97 & 0.0 & 0.0 & 0.00(a) & 0.0 & 1\\\hline
\multirow{3}{6em}{\centering{30\\(f1000g200a3)}} & RHF & 4922.02 & 4922.02 & 0.0 & 0.0 & 0.00(a) & 0.0 & 1\\
& RHM & 4844.84 & 4844.84 & 0.0 & 0.0 & 0.00(a) & 0.0 & 1\\
& SM+RHM & 3701.92 & 3701.92 & 0.0 & 0.0 & 0.00(a) & 0.0 & 1\\\hline
\multirow{3}{6em}{\centering{31\\(f1000g200a4)}} & RHF & 5897.09 & 5897.09 & 0.0 & 0.0 & 0.00(a) & 0.0 & 1\\
& RHM & 5715.54 & 5715.54 & 0.0 & 0.0 & 0.00(a) & 0.0 & 1\\
& SM+RHM & 3093.73 & 3093.73 & 0.0 & 0.0 & 0.00(a) & 0.0 & 1\\\hline
\multirow{3}{6em}{\centering{32\\(f1000g200a5)}} & RHF & 4913.89 & 4913.89 & 0.0 & 0.0 & 0.00(a) & 0.0 & 1\\
& RHM & 4654.29 & 4654.29 & 0.0 & 0.0 & 0.00(a) & 0.0 & 1\\
& SM+RHM & 2966.66 & 2966.66 & 0.0 & 0.0 & 0.00(a) & 0.0 & 1\\\hline
\multirow{3}{6em}{\centering{33\\(f1113g192a19)}} & RHF & 6305.36 & 6305.36 & 0.0 & 0.0 & 0.00(a) & 0.0 & 1 \\
& RHM & 6177.28 & 6177.28 & 0.0 & 0.0 & 0.00(a) & 0.0 & 1\\
& SM+RHM & 3865.39 & 3865.39 & 0.0 & 0.0 & 0.00(a) & 0.0 & 1\\\hline
\multirow{3}{6em}{\centering{34\\(f1095g192a20)}} & RHF & 5763.05 & 5763.05 & 0.0 & 0.0 & 0.00(a) & 0.0 & 1\\
& RHM & 6808.54 & 6808.54 & 0.0 & 0.0 & 0.00(a) & 0.0 & 1\\
& SM+RHM & 3420.51 & 3420.51 & 0.0 & 0.0 & 0.00(a) & 0.0 & 1\\\hline
\multirow{3}{6em}{\centering{35\\(f1092g192a21)}} & RHF & 7156.45 & 7156.45 & 0.0 & 0.0 & 0.00(a) & 0.0 & 1\\
& RHM & 8568.96 & 8568.96 & 0.0 & 0.0 & 0.00(a) & 0.0 & 1\\
& SM+RHM & 3603.5 & 3603.5 & 0.0 & 0.0 & 0.00(a) & 0.0 & 1\\\hline
\multirow{3}{6em}{\centering{36\\(f1125g192a22)}} & RHF & 8384.49 & 8384.49 & 0.0 & 0.0 & 0.00(a) & 0.0 & 1\\
& RHM & 9347.75 & 9347.75 & 0.0 & 0.0 & 0.00(a) & 0.0 & 1\\
& SM+RHM & 4279.68 & 4279.68 & 0.0 & 0.0 & 0.00(a) & 0.0 & 1\\\hline
\multirow{3}{6em}{\centering{37\\(f1125g192a23)}} & RHF & 7784.41 & 7784.41 & 0.0 & 0.0 & 0.00(a) & 0.0 & 1\\
& RHM & 7727.33 & 7727.33 & 0.0 & 0.0 & 0.00(a) & 0.0 & 1\\
& SM+RHM & 4325.52 & 4325.52 & 0.0 & 0.0 & 0.00(a) & 0.0 & 1\\\hline
\end{longtable}

\subsection{Summary of the computational results}
In summary, it seems that the dynamic programming algorithm together with the submodular maximization offer the best performance for the small-sized and moderate-sized problems. For the large-sized instances, it seems that running the submodular maximization algorithm when $\overline{\sigma}$ is large and using the rolling horizon method while keeping the horizon sizes at most $20$ otherwise lead to the best performances. Moreover, the binary program produces feasible solutions of good quality for almost all instances as the gaps are usually very small. The amount of time taken to obtain those feasible solutions is also reasonable in practice. 

Finally note that all instance data can be accessed online at the authors' websites.   

\section{Concluding remarks} \label{conclusions}
We use a column generation approach to solve the flight-to-gate assignment problem aimed at minimizing the total arrival delays. Instead of using the integer program solver for the pricing problem, more efficient approximation algorithms and dynamic programming algorithms are proposed. Among the proposed algorithms, submodular maximization algorithm shows very strong performances on the small and medium-sized instances and together with the rolling horizon framework, it allows us to obtain solutions of very good quality very efficiently for the large-sized instances.

There are few possible extensions to consider. One of them is to consider the interferences between different gates. With different airport layouts, this can be realized by linking constraints in the master problem or a decomposition of the pricing problem by gate groups instead of by each individual gate. Another extension is to take the uncertainties in the arrival times and turn times into consideration. A stochastic programming approach may be suitable to tackle this version of the problem.

\appendix
\section{Appendix A.}
This proof is given in \cite{subMax} and we modified it to match our assumption.\\
Let $OPT$ denote an optimal solution and $OPT_i := (OPT \cup X_i) \cap Y_i$, then we have that $OPT_0 = OPT$ and $OPT_n = X_n = Y_n$. Here are two useful lemmas to be used later in the proof of the theorem. 
\begin{lemma}\label{submodularApproxLemma1}
For every $1 \le i \le n$, $a_i + b_i \ge 0$.
\end{lemma}
\proof{Proof of Lemma \ref{submodularApproxLemma1}.}
Note that if $f$ is a submodular function with a ground set $\mathcal{F}'$, we have that if $A,B \subset \mathcal{F}'$, $f(A) + f(B) \ge f(A \cup B) + f(A \cap B)$.\\
Now, we see that $(X_{i-1} \cup \{u_i\}) \cup (Y_{i-1} \backslash \{u_i\}) = Y_{i-1}$ and $(X_{i-1} \cup \{u_i\}) \cap (Y_{i-1} \backslash \{u_i\}) = X_{i-1}$. Then by the above defintion, we have that 
\begin{align}
a_i + b_i & = [f(X_{i-1} \cup \{u_i\}) - f(X_{i-1})] + [f(Y_{i-1} \cup \{u_i\}) - f(Y_{i-1})]\\
& = [f(X_{i-1} \cup \{u_i\}) + f(Y_{i-1} \backslash \{u_i\})] - [f(X_{i-1}) + f(Y_{i-1})] \ge 0.
\end{align} 
\endproof

\begin{lemma}\label{submodularApproxLemma2}
For every $1 \le i \le n$, 
\begin{align}
\mathbb{E} [f(OPT_{i-1}) - f(OPT_i)] \le \frac{1}{2} \mathbb{E} [f(X_i) - f(X_{i-1}) + f(Y_i) - f(Y_{i-1})] \label{submodularMainExp}
\end{align}
\end{lemma}
\proof{Proof of Lemma \ref{submodularApproxLemma2}.}
It is sufficient to prove (\ref{submodularMainExp}) conditioned on any event of the form $X_{i-1} = S_{i-1}$, when $S_{i-1} = \{u_1, \cdots, u_{i-1}\}$ and the probability $X_{i-1} = S_{i-1}$ is non-zero. Hence fix such an event for a given $S_{i-1}$. The rest of the proof implicitly assumes everything is conditioned on this event. Note due to conditioning, the following quantities become constants,
\begin{itemize}
	\item $Y_{i-1} = S_{i-1} \cup \{u_i, \cdots, u_n\}$
	\item $OPT_{i-1} := (OPT \cup X_{i-1}) \cap Y_{i-1} = S_{i-1} \cup (OPT \cap \{u_i, \cdots, u_n\})$
	\item $a_i$ and $b_i$.
\end{itemize}
From Lemma \ref{submodularApproxLemma1}, we have $a_i + b_i \ge 0$, so we only need to consider three cases. 
\begin{enumerate}
	\item $a_i \ge 0$ and $b_i < 0$. In this case, $a_i^{\prime} / (a_i^{\prime} + b_i^{\prime}) = 1$. Then we have $Y_i = Y_{i-1}$ and $X_i = S_{i-1} \cup \{u_i\}$. Hence $f(Y_{i-1}) - f(Y_i) = 0$. Also, $OPT_i = (OPT \cup X_i) \cap Y_i = (OPT \cup X_{i-1} \cup \{u_i\}) \cap Y_{i-1} = OPT_{i-1} \cup \{u_i\}$. Then we are left to prove that  
	\begin{equation}
	f(OPT_{i-1}) - f(OPT_{i-1} \cup \{u_i\}) \le \frac{1}{2}[f(X_i) - f(X_{i-1})] = \frac{a_i}{2}.
	\end{equation}
	If $u_i \in OPT_{i-1}$, then the left hand side is zero and this inequality is satisfied. If $u_i \notin OPT_{i-1}$,  then we note that 
	\begin{equation}
	OPT_{i-1} = (OPT \cup X_{i-1}) \cap Y_{i-1} \subseteq Y_{i-1} \backslash \{u_i\}.
	\end{equation}
	Next, by the definition of submodularity of $f$, we have now
	\begin{equation}
	f(OPT_{i-1}) - f(OPT_{i-1}  \cup \{u_i\}) \le f(Y_{i-1} \backslash \{u_i\}) - f(Y_{i-1}) = b_i \le 0 \le \frac{a_i}{2}.	
	\end{equation}
	\item $a_i < 0$ and $b_i \ge 0$. This case is analogous to the previous case.
	\item $a_i \ge 0$ and $b_i > 0$. In this case, we have $a_i^{\prime} = a_i$ and $b_i^{\prime} = b_i$. Then we can compute the left hand side of (\ref{submodularMainExp}) by 
	\begin{align}
	\mathbb{E} [f(X_i) - f(X_{i-1}) + f(Y_i) - f(Y_{i-1})] & = \frac{a_i}{a_i + b_i} [f(X_{i-1} \cup \{u_i\}) - f(X_{i-1})] \notag \\
	& + \frac{b_i}{a_i + b_i} [f(Y_{i-1} \cap \{u_i\}) - f(Y_{i-1})]\\
	& = \frac{a_i^2 + b_i^2}{a_i + b_i} \label{a_i>0b_i>0_1}.
	\end{align}
	Next we upper bound $\mathbb{E} [f(OPT_{i-1}) - f(OPT_i)]$,
	\begin{align}
	\mathbb{E} [f(OPT_{i-1}) - f(OPT_i)] & = \frac{a_i}{a_i + b_i} [f(OPT_{i-1}) - f(OPT_{i-1}) \cup \{u_i\}] \notag \\
	& + \frac{b_i}{a_i + b_i} [f(OPT_{i-1}) - f(OPT_{i-1} \backslash \{u_i\})] \\
	& \le \frac{a_ib_i}{a_i + b_i} \label{a_i>0b_i>0_2}.
	\end{align}
	Last inequality follows from the following cases. Note that $u_i \in Y_{i-1}$ and $u_i \notin X_{i-1}$,
	\begin{enumerate}
	\item $u_i \notin OPT_{i-1}$, then we note that the second term of the left hand side the last inequality is zero and 
	\begin{equation}
	OPT_{i-1} = (OPT \cup X_{i-1}) \cap Y_{i-1} \subseteq Y_{i-1} \backslash \{u_i\}.
	\end{equation}
	Next, by the definition of submodularity of $f$, we have now
	\begin{equation}
	f(OPT_{i-1}) - f(OPT_{i-1}  \cup \{u_i\}) \le f(Y_{i-1} \backslash \{u_i\}) - f(Y_{i-1}) = b_i.
	\end{equation}
	\item $u_i \in OPT_{i-1}$, then we note that the first term of the left hand side the last inequality is zero and 
	\begin{equation}
	X_{i-1} \subseteq (OPT \cup X_{i-1}) \cap Y_{i-1} \backslash \{u_i\} = OPT_{i-1} \backslash \{u_i\}.
	\end{equation}
	Next, by the definition of submodularity of $f$, we have now
	\begin{equation}
	f(OPT_{i-1}) - f(OPT_{i-1} \backslash \{u_i\}) \le f(X_{i-1} \cup \{u_i\}) - f(X_{i-1}) = a_i.	
	\end{equation}
	\end{enumerate}
	Then with (\ref{a_i>0b_i>0_1}) and (\ref{a_i>0b_i>0_2}), we are left to verify
	\begin{equation}
	\frac{a_ib_i}{a_i + b_i} \le \frac{1}{2} \left(\frac{a_i^2 + b_i^2}{a_i + b_i} \right),	
	\end{equation}
	which can be easily verified. 
\end{enumerate}
\endproof

With the two Lemmas, we are ready to show the approximation guarantee of the algorithm. 
\proof{Proof of Theorem \ref{subguarantee}.}
If we sum all inequalities from lemma \ref{submodularApproxLemma2} over $1 \le i \le n$, we have that 
\begin{align}
& \sum_{i=1}^n \mathbb{E} [f(OPT_{i-1}) - f(OPT_i)] \le \frac{1}{2} \sum_{i=1}^n \mathbb{E} [f(X_i) - f(X_{i-1}) + f(Y_i) - f(Y_{i-1})]\\
\implies & \mathbb{E} [f(OPT_0) - f(OPT_n)] \le \frac{1}{2} \mathbb{E} [f(X_n) - f(X_0) + f(Y_n) - f(Y_0)] .\label{constantApproxExpr}
\end{align}
Note that $X_0 = \emptyset$, so we have $f(X_0) = 0$ and by our assumption we have $f(Y_0) = f(\mathcal{F}') \ge 0$. In addition, we have $OPT_0 = OPT$ and $OPT_n = X_n = Y_n$, so (\ref{constantApproxExpr}) gives that 
\begin{align*}
f(OPT) \le \frac{1}{2} \mathbb{E} [2f(X_n) - f(Y_0)] + \mathbb{E} f(X_n) \le \frac{1}{2} \mathbb{E} [2f(X_n)] + \mathbb{E} f(X_n) \implies \frac{1}{2} f(OPT) \le \mathbb{E} f(X_n).
\end{align*}
\endproof

\section{Appendix B.}
\subsection*{Rational input data for dynamic programming algorithm.} \label{intDynamicAlgo}
Recall a direct implementation of the dynamic programming algorithm terminates with a worst case $2^n$ recursive calls as proven in Proposition \ref{finitealgorithm}. We can improve the complexity significantly if we further assume all input data are rationals and run a backward version of the dynamic programming algorithm. Input data includes all input parameters to the pricing problems, namely the arrival times, $a_i$, the dual information, $\pi_i$, and the processing time, $p_i$. Without loss of generality, we assume the input data are expressed as rationals with a common denominator. We introduce a new notion of last time to consider for the pricing problems, denoted by $c$. This is the time beyond which no further flights are accepted. It suffices to choose $c$ as $a_n + \max_{1 \le i \le n} \pi_i$. For any flight $j$, if we consider the function $g_{j}(t)$ with $t > c$, we get that $t > a_l + \pi_l$ for $j \le l \le n$. Consequently the function value $g_j(t) = g_l(t) = g_{n+1}(t) = 0$ and $x_{lk} = 0$ for $j \le l \le n$.

For every $i \in [n]$, we see that the function $g_i(t)$ is piecewise linear and we label the points where the slope of the function changes as the breakpoints of the function. When we have integral input data, each $g_i(t)$ does not have too many breakpoints. The following theorem gives a formal statement about this observation. 
\begin{theorem}\label{breakpoints}
Assume all input data are rational and c is last time to consider. Let the common denominator be $d$, then each function $g_i(t)$ for $i \in [n]$ has at most $O(n^2dc)$ breakpoints.  
\end{theorem}
To prove Theorem \ref{breakpoints}, we first consider the possible slopes for a particular function $g_{n-i}(t)$.   
\begin{lemma} \label{fslope}
Given a function $g_{n-i}(t)$, its slope can only take values from the set $\{0,-1,\cdots, -i-1\}$
\end{lemma}
\begin{proof}[proof of Lemma \ref{fslope}.]
We prove this lemma by backward induction. We observe that the function $g_n(t)$ is given by
\begin{align}
g_n(t) = \begin{cases}\pi_n \;\; & \text{ if }t \le a_n\\
a_n + \pi_n - t\;\; & \text{ if }a_n < t \le a_n + \pi_n \\
0\;\; & \text{ if }t > a_n + \pi_n.
\end{cases}
\end{align}

The set of slopes of function $g_n(t)$ is $\{0,-1\}$. This is the base case of $n$. Suppose the statement is true for $n$, $n-1$, $\cdots, n-i$. Then the set of possible slopes of $g_{n-i}(t)$ is $\{0,-1,\ldots,-i-1\}$. Now consider the function $g_{n-i-1}(t)$, based on the recursive formula (\ref{recureq}), 
\begin{align}\label{fslopeeq}
g_{n-i-1}(t) = \begin{cases}
g_{n-i-1}(a_i) & \text{if } t \le a_{n-i-1}\\
g_{n-i}(t) & \text{if } t > a_{n-i-1} + \pi_{n-i-1}\\
\max\{a_{n-i-1} + \pi_{n-i-1} - t + g_{n-i}(t + p_{n-i-1}), g_{n-i}(t)\} & \text{if }a_{n-i-1} < t \le a_{n-i-1} + \pi_{n-i-1}.
\end{cases} 
\end{align}

In the second case, the set of possible slopes of $g_{n-i-1}(t)$ is the same as that of $g_{n-i}(t)$. In the third case, $g_{n-i-1}(t)$ can have an additional slope than $g_{n-i}$ does. If $g_{n-i}(t+p_{n-i-1})$ has a slope of $-i-1$, the slope of $g_{n-i-1}(t)$ can be $-i-2$. Therefore, the set of possible slopes of $g_{n-i-1}$ is $\{0,-1,\ldots, -i-2\}$. Thus the statement is also true for $n-i-1$, which completes the proof.
\end{proof}

Intuitively, the slope of the function corresponds to the rate of loss in the total net benefits if we delay the park times of flights. For the flight $n-i$, it can only potentially affect the park times of itself and flights after it, i.e. from the set $\{n-i, n-i+1, \ldots, n\}$ and the corresponding function $g_{n-i}(t)$ can have slopes from the set $\{0,-1,\ldots,-i-1\}$.

Now we are ready to give a proof of theorem \ref{breakpoints}. 

\begin{proof}[proof of Theorem \ref{breakpoints}.]
Consider a function $g_{n-i}(t)$, from the recursive formula (\ref{recureq}), we observe that all the potential breakpoints can be computed by equating the two components in the max expression. The solution of $t$ may constitute a breakpoint of $g_{n-i}$. As a result, we consider the following linear equation
\begin{equation} \label{breakpointseq}
(a_{n-i} + \pi_{n-i} - t) + g_{n-i+1}(t + p_i) = g_{n-i+1}(t).
\end{equation}

Functions $g_{n-i+1}(t + p_i)$ and $g_{n-i+1}(t)$ are piecewise linear in variable $t$. We can write 
\begin{equation}
g_{n-i+1}(t + p_i) = c_1 - k_1t \text{ and } g_{n-i+1}(t) = c_2 - k_2t,
\end{equation}
where $c_1$ and $c_2$ are rational constants and $k_1$ and $k_2$ are integer constants. We only consider cases where $k_1 + 1 \neq k_2$, otherwise the two functions on the left hand side and right hand side of (\ref{breakpointseq}) have the same slope and do not constitute a breakpoint. Now equation (\ref{breakpointseq}) can be rewritten as 
\begin{equation}
a_{n-i} + \pi_{n-i} - t + c_1 - k_1t = c_2 - k_2t \implies t = \frac{|a_{n-i} + \pi_{n-i} +c_1 - c_2|}{|k_1 + 1 - k_2|}.
\end{equation}

As shown in lemma \ref{fslope}, $k_1$ and $k_2$ only take values from the set $\{0,-1,\ldots,-i\}$ and thus $|k_1 + 1 - k_2|$ can only take values from the set $\{1,\ldots,i+1\}$. In addition, since the input data are rationals with a common denominator $d$, the numerator $a_{n-i} + \pi_{n-i} + c_1 - c_2$ can be expressed as $C/d$ where $C$ is an integer. Then all the breakpoints of function $g_{n-i}(t)$ are in the set
\begin{align} \label{bpform}
\{0,e,2e,3e,\ldots,c\} \text{ where }e \in \{1/d,1/2d,1/3d,\ldots,1/(i+1)d\}.
\end{align}

The total number of breakpoints is then at most $1dc + 2dc + 3dc + \cdots + (i+1)dc = O(n^2dc)$, which completes the proof. 
\end{proof}

Note that since each of the functions $g_i(t)$ has at most $O(n^2dc)$ breakpoints, we only need to evaluate the function at these potential breakpoints to construct the function and consequently we need to run in $O(n^3dc)$ to construct all functions $g_i(t)$ in the interval $[0,c]$. In particular, note that only the function values at the points in the list of the potential breakpoints are needed since the functions are all piece-wise linear. We can first construct the function $g_n(t)$ in the interval $[0,c]$ by evaluating the function at the points in the list of the potential breakpoints given in (\ref{bpform}) in reverse order and then we construct the function $g_{j}(t)$ for $1 \le j \le n-1$ using previously constructed functions $g_l(t)$ for $j+1 \le l \le n$ based on the recursive formula (\ref{recureq}).

Although the above analyses provide a simpler complexity in the case of rational input data, it is not very effective in practice as the common denominator $d$ can be large after scaling the input data.

\section{Appendix C.}
\subsection*{Rolling horizon method.} \label{rollinghorizon}
The rolling horizon framework can also be utilized to decompose the pricing problems to reduce their sizes. As with any implementation of the rolling horizon method, we must decide on two parameters. One is the horizon size and the other is the window size. The horizon size can be either fixed or dynamically determined for each gate and each iteration. The adjacency parameter given in Definition \ref{adjacencyparameter} is an example of possible dynamically determined horizon sizes. We give a brief implementation of the rolling horizon method in Algorithm \ref{rollinghorizonalgo} that we have implemented in our experiments.
\begin{algorithm}[H]
\caption{Rolling horizon} \label{rollinghorizonalgo}
\begin{algorithmic}[1]
\State \textbf{Input:} horizon size, $l$, and window size, $w$
\State Initilize $s = 0$, $t= 0$, and  $S = \emptyset$
\While{$s + l \le n$}
\State Solve for the optimal assignment for the flights $\{s+1,s+2,\ldots, s+l\}$ at time $t$
\State Fix the assignments of first $w$ flights and add them to $S$
\State $t \gets \text{time at which the gate becomes available under the set of assigned flights } S$
\State $s \gets s + w$
\EndWhile
\State Solve for the optimal assignment for the flights $\{s+1,s+2,\ldots, n\}$ at time $t$
\State Fix the assignments of the flights $\{s+1,s+2,\ldots, n\}$ and add them to $S$ 
\State Compute the total net benefits of the assignment $S$, $f(S)$
\State \Return{$S$, $f(S)$}
\end{algorithmic}
\end{algorithm}
Similar to the block decomposition approximation algorithm, we refer to the discussions in Subsection \ref{approximation} for the computation of the total net benefits $f(S)$ and the time at which the gate becomes available under a set of assigned flights. 

\newpage
\bibliographystyle{plain}
\bibliography{GA.bib}

\end{document}